\documentclass[11pt,reqno,a4paper]{amsart} 
\usepackage{amsmath}
\usepackage{amssymb, latexsym, amsfonts, amscd, amsthm, 
mathrsfs, enumerate, esint}

\usepackage{accents}

\usepackage[usenames,dvipsnames]{color}
\usepackage[all]{xy}
\usepackage{graphicx}
\usepackage{epsfig}
\usepackage{hyperref}
\usepackage{marginnote}
\usepackage{mathtools} 			
\usepackage{bbm}
\usepackage{bm}
\usepackage{amssymb}
\usepackage{amsmath}
\usepackage{amsfonts,latexsym,amstext}
\usepackage{enumitem}

\usepackage{pifont}
\usepackage[numbers,sort]{natbib}

\numberwithin{equation}{section}

\definecolor{purple}{rgb}{0.9,0,0.8}

\definecolor{gray}{rgb}{0.7,0.7,0.7}

\newtheorem{theorem}{Theorem}[section]

\newtheorem{lemma}[theorem]{Lemma}
\newtheorem{proposition}[theorem]{Proposition}

\theoremstyle{definition}
\newtheorem{definition}[theorem]{Definition}

\newtheorem{remark}[theorem]{Remark}

\usepackage[top=2cm, bottom=2cm, left=2cm,right=2cm]{geometry}

\newcommand{\beq}{\begin{equation}}
\newcommand{\eeq}{\end{equation}}

\newcommand{\N}{\mathbb{N}}

\newcommand{\R}{\mathbb{R}}

\newcommand{\Z}{\mathbb{Z}}

\DeclareMathOperator{\argmin}{arg\,min}

\renewcommand{\emptyset}{\varnothing}

\renewcommand{\setminus}{\backslash}

\newcommand{\CAL}[1]{\mathcal{#1}}  
\newcommand{\BB}[1]{\mathbb{#1}}

\def\mathcolor#1#{\@mathcolor{#1}}
\def\@mathcolor#1#2#3{%
	\protect\leavevmode
	\begingroup
	\color#1{#2}#3%
	\endgroup
}
\newcommand{\blanco}[1]{\mathcolor{white}{#1}}

\def\cle{\preccurlyeq}

\def\nn{\nonumber}

\title[Brunet-Derrida particle systems]{Rank Dependent Branching-Selection Particle Systems}
\author[P. Groisman]{Pablo Groisman}
\address{Universidad de Buenos Aires, Facultad de Ciencias Exactas y Naturales and IMAS-CONICET, Argentina; NYU-ECNU Institute of Mathematical Sciences at NYU Shanghai.}
\email{pgroisma@dm.uba.ar}
\author[N. Soprano-Loto]{Nahuel Soprano-Loto}
\address{Universidad de Buenos Aires, Facultad de Ingenier\'ia,  Argentina.}
\email{nsoprano@fi.uba.ar}

\begin{document}


\begin{abstract}
We consider a large family of branching-selection particle systems. The branching rate of each particle depends on its rank and is given by a function $b$ defined on the unit interval. There is also a killing measure $D$ supported on the unit interval as well. At branching times, a particle is chosen among all particles to the left of the branching one by sampling its rank according to $D$. The measure $D$ is allowed to have total mass less than one, which corresponds to a positive probability of no killing. Between branching times, particles perform independent Brownian Motions in the real line. This setting includes several well known models like Branching Brownian Motion (BBM), $N$-BBM, rank dependent BBM, and many others. We conjecture a scaling limit for this class of processes and prove such a limit for a related class of branching-selection particle system. This family is rich enough to allow us to use the behavior of solutions of the limiting equation to prove the asymptotic velocity of the rightmost particle under minimal conditions on $b$ and $D$. The behavior turns out to be universal and depends only on $b(1)$ and the total mass of $D$. If the total mass is one, the number of particles in the system $N$ is conserved and the velocities $v_N$ converge to $\sqrt{2 b(1)}$. When the total mass of $D$ is less than one, the number of particles in the system grows up in time exponentially fast and the asymptotic velocity of the rightmost one is $\sqrt{2 b(1)}$ independently of the number of initial particles.
\end{abstract}

\maketitle

\section{Introduction}

Branching-selection particle systems have been widely studied for a long time. They are useful to model the evolution of a population under selection mechanisms but also in chemistry, physics and other branches of biology since they are good microscopic versions for phenomena that at a large scale show the propagation of a front between a stable and an unstable state. This is a common situation in all these disciplines and many others.

Since the seminal paper by Brunet and Derrida \cite{BD2}, many models have been introduced to describe and understand the differences between microscopic and macroscopic models through heuristic arguments, numerical simulations and rigorous proofs \cite{BD,BD2,BDMM,BDMM2,S,BBP,Mai13,Mai16,DR,DFPS}. 

Several properties of the system at the microscopic level have been conjectured -and sometimes also proved- to be universal among theses models, like the shift in the velocity of the front and the asymptotic expansion of the rate of convergence of the microscopic velocities to the macroscopic one.

In this article we introduce a family of models that can be considered to belong to the Brunet-Derrida class
and contain some well known models as particular instances. We prove the existence of an asymptotic velocity for all of them and the convergence of these velocities to the universal constant $\sqrt 2$ as the number of particles increases to infinity. The main tool is a rigorous proof of the scaling limit of suitable processes to an F-KPP type equation. 
The strategy of using the hydrodynamic limit to get information about the particle system have been widely used in different contexts to understand random walks and particle systems \cite{ABP, GJM, GQ, BCDM+86}.

The novelty here is that although we are not able to prove the scaling limit for all the instances of the model, the class of processes for which we are able to prove it is rich enough to allow us to show the convergence of the velocities in all the cases. In addition, we provide heuristic arguments to conjecture the hydrodynamic equation for any choice of $b$ and $D$. The hydrodynamic limit equation encodes both F-KPP type equations as well as free-boundary problems like the ones in \cite{GJM,DFPS,BBP} as particular cases, but also many others, including the possibility of non local terms.

\subsection{\texorpdfstring{$(b,D)$}{(b,D)}-Branching Brownian Motion} 
We first describe the model in words. The parameters are a {\em birth function} $b\colon[0,1] \to [0,\infty)$ 
and a {\em death probability measure} supported on $\mathbb I:=\{-\infty\}\cup[0,1)$ defined through its cumulative distribution function $D\colon \mathbb R \to [0,1]$. 
For simplicity, we assume $b(1)=1$.
The evolution is given by a continuous-time Markov process that performs independent Brownian Motions on the real line except at branching times, at which particles can branch into two (reproduction) and can also be eliminated from the system (selection mechanism). 
At time $ t=0 $ we start with a deterministic number of particles $ N $
whose positions in the line may be given by any distribution on $ \mathbb R^N $.
Let $N_t$ be the number of particles in the system at time $t$. 
We use $ X^N_t(1),\ldots,X^N_t(N_t) $ to denote their positions at that time. 
For $j=1, \dots, N_t$, the particle with quantile $j$ branches into two particles at rate $ b\big( \frac{j-1}{N_t-1} \big)$. 
Hereafter we abuse a little bit and use the word quantile as a synonym for order statistic.
At the time a particle with quantile $j$ branches, the particle with quantile $ i\in \{ 1,\ldots,j-1 \}$  is killed with probability $ D(\frac{i}{j-1}-) - D(\frac{i-1}{j-1}-) $. Observe that the number of particles $N_t$ is constant if $ D(-\infty)=0 $, but there is a positive probability of no killing at a branching time if $D(-\infty)>0$. In that case the number of particles in the system increases exponentially fast.
The notations $ D(x-)$ and $ D(-\infty)$ stand for $\lim_{y\uparrow x}D(y) $ and $\lim_{y\to -\infty}D(y) $ respectively.
A graphical construction of the $ (b,D) $-BBM is provided in the course of the proof of Proposition \ref{thm_comparison}.

We will first discuss the relevance of this model and we will compare it with well known processes in the Brunet-Derrida class that have been previously studied, some of which can be obtained as particular instances for adequate choices of $b$ and $D$. Then, we review the main properties of the $(b,D)$-BBM when the number of particles is conserved, $D(-\infty)=0$. This is in the spirit of \cite{DR,DM,GJM} and there is no new ideas here. The important fact is that the process as seen from the tip is ergodic and that this implies the existence of an asymptotic velocity $v_N > 0$ for the cloud of particles,
\[
\lim_{t\to\infty} {t^{-1}}\displaystyle\max_{1\le i \le N} X^N_t(i) = v_N.
\]
Afterwards we study the scaling limit of the process as the number of initial particles $N$ goes to infinity. As a byproduct, we obtain the convergence of the velocities. 

A proof of the scaling limit for general $b$ and $D$ is out of the scope of this paper. To get an idea of the level of difficulty of the problem, it is worth to note that while for absolutely continuous (with respect to Lebesgue) measures $D$ we expect a nice reaction-diffusion equation, as in \cite{GJM}, while a free-boundary is expected to be involved in the formulation of the hydrodynamic equation when $D$ has an atom at zero, as in $N$-BBM \cite{DFPS,BBP}. The main obstacle is the lack of a proof of propagation of chaos for such general $b$ and $D$, but we will see that we can obtain nice bounds for the two-particle correlations for a {large class of processes that are related to any $(b,D)$-BBM}. Once this is obtained, the control of the variance of the empirical measures follows readily and with the help of proper comparison principles, we can get our result. 


This leads us to the following result.

\begin{theorem}\label{thm_main}
	Let $ X^N=\{X^N_t:t\ge 0\} $ be the $ (b,D) $-BBM  with arbitrary random initial condition $ X^N_0\in\mathbb R^N $.
	Suppose that there exists $ k\in\mathbb N $
	such that $ x^k\le b(x)\le 1 $ and $ x^k\le D(x) $ for every $ x\in [0,1] $.
	\begin{enumerate}
		\item If $ D(-\infty) =0$,
		$ X^N $ has a deterministic asymptotic velocity that depends only on the number of particles $ N $. There exists $ v_N > 0$ such that 
		\begin{align}\nonumber
		\lim_{t\to\infty} {t^{-1}}\displaystyle\max_{1\le i \le N} X^N_t(i) = v_N \quad\mbox{a.s and in $ L^1 $}.
		\end{align}
		Furthermore, 
		\[
		\lim_{N\to\infty}v_N=\sqrt 2.
		\]
		\item If $ D(-\infty)>0 $,
		the asymptotic velocity of  $ X^N $ is $ \sqrt 2 $ for every $N$,
		\begin{align}\nonumber
		\lim_{t\to\infty} {t^{-1}}\displaystyle\max_{1\le i \le N} X^N_t(i) = \sqrt{2} \quad\mbox{a.s and in $ L^1 $}.
		\end{align}
	\end{enumerate}
\end{theorem}

\begin{remark}
	Since $ \lim_{k\to\infty}x^k=0 $ for $ x\in [0,1) $, the existence of such an integer $ k $ is  a mild requirement.
	The assumption $ b(1)=1 $ is imposed just to normalize and can easily be removed. In that case we need to assume $ b(1)x^k\le b(x) $ instead and we get the asymptotic velocity $ \sqrt{2b(1)} $.
	
\end{remark}

\begin{remark}
	In \cite{BD2}, based on numerical simulations and heuristic arguments, E. Brunet and B. Derrida conjectured the unexpected slow rate of order $(\log N)^{-2}$ for the convergence of $v_N$ to $\sqrt 2$ in (1) for a related model of $N$ interacting particles in $\Z$ that evolve at discrete times and suggested that this behavior should be universal. The conjectured universality comes from the fact that the behavior of $\sqrt 2 - v_N$ can be understood by introducing a cut-off to the scaling limit of the system. Remarkably, this order of convergence has been proved rigorously for some models \cite{BG10,P}. We expect the same order of convergence in our setting at least for a large class of parameters $b$ and $D$, but the exact class of pairs $(b,D)$ for which this order of convergence holds is a delicate question. Answering this question should involve a careful study of the solutions of equation \eqref{hydro.limit} below with a cut-off. This research has not been carried out yet.
\end{remark}

\section{Relevance of the Model and Related Work}
\label{relevance}
The $(b,D)$-BBM is a natural model for the evolution of a genetic trait in the presence of selection and similar phenomena. In fact, it certainly fits in the spirit of all the models introduced by Brunet, Derrida and coauthors in their seminal papers \cite{BD, BD2, BDMM, BDMM2}. We will see now that several models that have been studied in the literature can be obtained as particular cases of the $(b,D)$-BBM for adequate choices of $b$ and $D$.

Before going into that, it is worth mentioning that systems of diffusing particles interacting through their ranks have also attracted the attention of scientists in probability, finances and many other areas \cite{CDSS,DT, DJ, IKS}, and it is known that several common features appear in this type of systems.

Also, branching-selection particle systems in which the killing rates depend on a fitness function have been studied \cite{Beckman,BZ}, and precise information on their behavior have been proved. 
In these models the fitness function depends on the absolute position of the particles rather than its relative one. In \cite{Beckman} a branching rate depending on the position and the empirical measure is considered and the hydrodynamic limit is obtained, but that setting is different to ours and also the technique. Finally, the Brownian Bees model have been considered recently in \cite{BBNP,ABLT}. In this model particles perform independent Brownian Motions in $\R^d$ and branch at rate one. At branching events the particle which is the furthest away from the origin is removed. In \cite{ABLT} in fact the killed particle is the one which is furthest from the barycenter instead of the origin and an invariance principle is obtained for every fixed number of particles under diffusive scaling,
while in \cite{BBNP} the hydrodynamic limit is obtained as $N\to\infty$ for i.i.d. initial conditions and, remarkably, also for the cloud of particles in equilibrium.

We list below the announced particular cases of the $ (b,D) $-BBM.
We make a slight abuse of notation by allowing $ D $ to denote also the probability measure supported on $ \BB I $ having $ D $ as cumulative distribution function when this does not lead to confusion.
\begin{enumerate}
	\item
	\label{BBM}   Taking $ b\equiv 1 $ and $ D=\delta_{-\infty} $,
	we get the BBM.
	\item \label{N-BBM}  The case $ b=\mathbbm{1}_{ (0,1] }$ and $ D=\delta_0 $ results in the $N$-BBM \cite{Mai13,Mai16,DFPS}.
	\item \label{fittest}
	Taking
	$b(s)=s$ and $ D=\text{Unif}([0,1]) $,
	we recover the model introduced in \cite{GJM}, in which particles diffuse as independent Brownian Motions. In addition every pair of particles is chosen at a constant rate $\frac{1}{N-1}$ and the leftmost one (among the chosen particles) jumps on top of the rightmost one. In fact,
	the $j$-th quantile belongs to $(j-1)$ pairs in which a particle will jump on top of it, so it branches at rate $b\big(\frac{j-1}{N-1}\big) = \frac{j-1}{N-1} $.
	On the other hand, conditioning on the event that a particle with quantile $j$ branches, the probability of a particle with quantile $i<j$ being part of the pair is uniform over the set $ \{ 1,\ldots,j-1 \} $. Namely, the probability is $ D\big( \frac{j}{i-1}\big) - D\big(\frac{j-1}{i-1} \big) =\frac{1}{i-1}$.
	\item   $b(s)\equiv 1$ and $D=\text{Unif}([0,\varepsilon])$ leads to interesting models as well. On the one hand they can be seen as smooth approximations of $N$-BBM, in the sense that the hydrodynamic equation has no free-boundaries and is just an F-KPP type equation. On the other hand, if we allow $\varepsilon$ to be random (which is not considered in this article), we get slight modifications of the very well known models of BBM with absorption \cite{S, BLS, GHS,AHZ}
	by taking $\varepsilon$ equal to the proportion of particles below some barrier,
	and we get a variant of the $L$-BBM model considered in \cite{BDMM,P}
	for $\varepsilon$ equal to the proportion of particles whose distance to the rightmost one is larger than $L$.
	\item  
	$b\equiv 1$ and $D(s)=1- ks(1-s)^{k-1} - (1-s)^k$ for $0\le s \le 1$ gives another smooth approximation of $N$-BBM as $k\to \infty$. These approximations have been considered in \cite{BBP} at the level of the hydrodynamic equation to prove the existence of solution of the free-boundary problem obtained as the scaling limit of $N$-BBM.
	We will discuss this example again just before Section \ref{sec_scaling_limit}.
	\item 
	If we take
	$b(s)=\mathbbm{1}\{s>\frac12\} $, or any other piecewise constant function, and any choice for $D$, we recover the model proposed in \cite{Beckman} replacing the mean by the median. Motivated by this model the author studies the hydrodynamic limit of a related BBM with branching rates depending on the position of the particle and the empirical measure in a specific way.  
\end{enumerate}

Let $F^{N}$ be the distribution function of the empirical measure of the particles\break$ X^N_t(1),\ldots,X^N_t(N_t)$, normalized by $N$,
\[
F^N(t,x)= \frac{1}{N}\sum_{i=1}^{N_t} \mathbbm{1}\{X^N_t(i) \le x\}.
\]
For $ t>0 $ fixed, and assuming the convergence of the initial conditions,
$F^N(t,\cdot)$ is expected to converge, as $ N\to\infty $, to a cumulative distribution function $U(t,\cdot)$ with density $u(t,\cdot)$ and tail distribution $V(t,\cdot)=1-U(t,\cdot)$.
We describe below these scaling limits in some of the situations already mentioned.

In case \eqref{BBM}, we get the heat equation with a source
\[
\partial_t u = \frac12 \partial_{xx}u + u,
\]
and, by linearity, the same equation for $U$. 

In case \eqref{N-BBM}, a free-boundary problem is obtained: find $(u,L)$ such that
\begin{align*}
\partial_t u  = \frac{1}{2} \partial_{xx}u + u, &\quad t>0, x\in (L_t, \infty)
\\
u(t,x) =0, &\quad t>0, x\in (-\infty,L_t)
&
\\
\quad \int_{L_t}^\infty u (t, x) \, dx = 1,
&\quad t>0,
\end{align*}
see \cite{GJ, DFPS, L, BBP, DR}.
Equivalently, integrating with respect to the spatial variable, the following equation was obtaind in \cite{BBP} for the tail distribution: find $(V,L)$ such that
\begin{align}
\nonumber
\partial_t V  = \frac{1}{2} \partial_{xx}V + V, &\quad t>0, x\in (L_t, +\infty)\\
\label{FreeBoundary.Tail.2} V(t,x) =1, &\quad t>0, x\in (-\infty,L_t),\\
\nonumber \partial_x V(t,L_t) =0,& \quad t>0.
\end{align}

For the case \eqref{fittest}, in \cite{GJM} the F-KPP was obtained for $U$:
\[
\partial_t U = \frac{1}{2} \partial_{xx}U - U(1-U).
\]
Differentiating with respect to the spatial variable readily gives the equation for the density $u$.

We end the section by discussing in a heuristic manner the most general form of the scaling limit.
For the $(b,D)$-BBM  ---that contains all the previous cases--- we expect,
when $ D $ has density $ d $,
the hydrodynamic equation to have the form
\begin{equation*}
\partial_t u  =\frac{1}{2}\partial_{xx} u+b({U})u-u\left[\int_{U}^1  b(r)\frac{1}{r}d \left({\frac{U}{r}}\right) {\mathrm dr}\right],  \quad  t>0,x\in \R.
\end{equation*}
Here both $u$ and $U$ are evaluated at $(t,x)$.
This equation has the following interpretation in terms of the rate at which particles are being created/eliminated at each position $x\in \R$: 
the first term corresponds to the diffusion of the particles; 
the second one follows since a particle at position $x$ branches at rate $b(U(x))$;
finally, to explain the third one, we observe that for a particle at position $x$ being eliminated we need, on the one hand, a particle to its right (higher quantile) to branch and, on the other hand, 
the involved particle to be chosen to die, 
this last choice being made through the measure $D$ rescaled to $[0,r]$ when the branching particle is the $r$-th quantile. 
By changing variables, we obtain the following formulation that does not require $D$ to have a density:
\begin{equation}
\label{hydro.limit}
\partial_t u  =\frac{1}{2}\partial_{xx} u+b({U})u-u\left[\int_{U}^1  b\left(\frac{U}{r}\right)\frac{1}{r} {D( \mathrm d r)}\right],  \quad t>0, x\in \R.
\end{equation}
Integrating on both sides with respect to the spatial variable, we get the equation
\begin{equation*}
\partial_t V  =\frac{1}{2}\partial_{xx} V+B({V})-G(V),  \quad t>0, x\in \R
\end{equation*}
for the tail distribution $V=1-U$,
being 
\[
B(z)=\int_{1-z}^1 b(s) \mathrm{d} s\quad\text{and}\quad G(z)= \int_{1-z}^1 \int_{s}^1  b\left(\frac{s}{r}\right)\frac{1}{r} {D( \mathrm d r)} \mathrm{d} s.
\]
Taking $b = \mathbbm{1}_{(0,1]}$ and $D=\delta_0$, we get $G(z)=\mathbbm{1}\{z=1\}$.
To see why this is true, consider the approximation $ \delta_0(\mathrm{d} r)\approx h^{-1}\mathbbm{1}\{ 0\le r\le h \}\mathrm{d} r $ as $ h\approx 0 $.
The integral
\begin{align}
\int_{1-z}^1 \int_{s}^1  \frac{1}{r}h^{-1} \mathbbm{1}\{0\le r\le h\} \mathrm{d} r\mathrm{d} s
\end{align}
vanishes when $ h $ is small enough if $ z<1 $,
and is constantly $ 1 $ if $ z=1 $.
Hence this choice of $ b $ and $ D $ results in
\begin{equation*}
\partial_t V  =\frac{1}{2}\partial_{xx} V+ V - \mathbbm{1}\{V=1\},  \quad t>0, x\in \R.
\end{equation*}
This equation that has the same weak formulation than equations \eqref{FreeBoundary.Tail.2}.
%
Also in \cite{BBP}, the authors obtain the solution to that problem as the limit as $k\to \infty$ of the solutions to problem
\[
\partial_t V_k  =\frac{1}{2}\partial_{xx} V_k+ V_k - V_k^k, \qquad \text{ for } x\in \R \text{ and } t>0, 
\]
which corresponds to taking $b\equiv 1$ and
$D(r) = D_k(r) = 1-(1-r)^{k-1}[(k-1)r+1] $
for $r\in [0,1]$
(that gives $\frac{1}{r}D_k'(r) = \frac{1}{r}d_k(r)  = k(k-1)(1-r)^{k-2}$). 
Then, the family of $(b,D)$-BBMs also contains a sequence of processes with parameters $(1, D_k)$ that converge to $N$-BBM not only at the level of the hydrodynamic equations but also at the level of processes (i.e: $D_k \to \delta_0$ as $k\to \infty$).


\section{Scaling Limit} 
\label{sec_scaling_limit}
As mentioned before, we are not able to prove the scaling limit of the $(b,D)$-BBM for general $(b,D)$;
nevertheless, we can do it for a class of processes that is large enough to bound from below the asymptotic velocities in the general case. This class has nonempty intersection with the $(b,D)$-BBM family but neither of them is contained in the other one. We think this result is of independent interest.

We introduce a process for which the number of particles $N$ is conserved. Between branching times, the particles diffuse as independent Brownian Motions. At rate $ \lambda N $,
a subset of $ k $ elements $ \{ \ell_1,\ldots,\ell_k\}\subset \{ 1,\ldots,N\} $ is chosen uniformly at random.
We suppose without loss of generality that $ \ell_1<\ldots<\ell_k $.
Instantaneously, with probability $ p(i,j) $ the particle with quantile $ \ell_i $ jumps on top of the one with quantile $ \ell_j $. Here $p$ is a probability on $\{ (i,j)\colon 1\le i < j \le k \} $.
For technical reasons, we allow particles to be located at $ -\infty $.
We call $(N,p)$-BBM a process with this distribution.

For a particle configuration $ \zeta\in [-\infty,\infty)^N $, we consider the distribution function of the associated empirical measure
\begin{align}\label{defi0001}
F_\zeta(x)=\frac{1}{N}\sum_{i=1}^N\mathbbm{1}\{ \zeta(i)\le x \}, x\in\R.
\end{align}
Let also $ h_p:[0,1]\to\BB R $ be the function defined by
\begin{align}\label{defi0002}
h_p(v) = \lambda\sum_{r=1}^{k-1} \hat p(r)\binom{k}{r}v^r(1-v)^{k-r},
\end{align}
where 
\begin{align}\label{p_hat}
\hat p(r)=\sum_{i\le r}\sum_{j> r}p(i,j) .
\end{align}
The coefficient $ \hat p(r) $ represents the probability of a particle with quantile smaller or equal than $ r $ jumping on top of a particle with quantile strictly larger than $ r $. We have the following hydrodynamic limit.

\begin{theorem}\label{thm_hydro}
	Fix $ k\ge 2 $, $ \lambda>0 $ and a probability $ p=(p(i,j))_{1\le i<j\le k} $.
	For every $ N\ge k $, 
	let $ \{Y^N_t:t\ge 0\} $ be the $(N,p)$-BBM
	with parameters $ \lambda>0 $ and $p$.
	Suppose that the initial distributions satisfy
	\begin{align*}
	\lim_{N\to\infty}\|F_{Y^N_0}-U_0 \|_\infty=0 \mbox{ in probability},
	\end{align*}
	being $ U_0 $ the distribution function of a probability on $[-\infty,\infty)$.
	Then, for every $ t>0 $,
	\begin{align}\label{reposera}
	\lim_{N\to\infty}
	\| F_{Y_t^N}-U(t,\cdot) \|_\infty
	=0 \mbox{ in probability},
	\end{align}
	where $U$ is the unique bounded solution of the F-KPP equation
	\begin{align}\label{kpp_eq_1}
	&\partial_t U=\tfrac{1}{2}\partial_{xx}U-h_p(U),\\ \label{kpp_eq_2}
	\nonumber	& U(0,\cdot)=U_0.
	\end{align}
\end{theorem}

We list below some interesting particular cases of the $ (N,p) $-BBM and their hydrodynamic equations.

\begin{enumerate}
	\item
	The particle with quantile ${k-1} $ deterministically jumps on top of the one with quantile $k $.
	This corresponds to $ p(i,j)=\mathbbm{1}\{ i=k-1 \} \mathbbm{1}\{ j=k \}$,
	leading to
	$ \hat p(r)=\mathbbm{1}\{r=k-1\} $
	and $ h_p(v)=\lambda kv^{k-1}(1-v) $.
	If we take $ \lambda=\frac{1}{k} $,  \eqref{kpp_eq_1}  reads
	\begin{align}\nonumber
	\partial_t U=\tfrac{1}{2}\partial_{xx}U-U^{k-1}(1-U).
	\end{align}
	This case is important because
	we are going to bound any $(b,D)$-BBM by one of this processes by choosing $k$ large enough.
	Observe that for $k=2$ the standard F-KPP equation is obtained.
	This scaling limit has been proved in \cite{GJM}. In fact our proof of Theorem \ref{thm_hydro} is a
	non-trivial generalization of the proof appearing there.
	\item 
	The particle with smallest position jumps on top of the one with largest position, i.e. $ p(i,j)=\mathbbm{1}\{ i=1 \} \mathbbm{1}\{ j=k \}$.
	This results in
	$ \hat p(r)=1 $ for every $ 1\le r\le k-1 $,
	and $ h_p(v) =\lambda(1-v^k-(1-v)^k)$.
	Taking $ \lambda=1 $ we obtain the equation
	\begin{align}\nonumber
	\partial_t U=\tfrac{1}{2}\partial_{xx}U-(1-U^k-(1-U)^k).
	\end{align}
	In the limit as $k\to \infty$ we get the free-boundary problem
	\begin{align}\nonumber
	\partial_t U=\tfrac{1}{2}\partial_{xx}U-\mathbbm{1}\{0<U<1\}.
	\end{align}
	\item The particle with smallest position jumps on top of a uniformly chosen one.
	This is $ p(i,j)=\frac{1}{k-1}\mathbbm{1}\{ i=1 \} $, giving $\hat p(r)=(k-r)/(k-1)$.
	Taking $ \lambda=1 $, we obtain
	\begin{align}\nonumber
	\partial_t U=\tfrac{1}{2}\partial_{xx}U-(1-U-(1-U)^k).
	\end{align}
	As already mentioned, 
	this equation has been used in \cite{BBP} as a smooth approximation to prove the existence of solution of the concerned free-boundary problem.
	Observe that if we  take $k=N$ (allowing $k$ to depend on $N$, which is not covered in our theorem), we obtain the $N$-BBM.
	\item 
	Fix a continuous function $ h:[0,1]\to [0,\infty) $ satisfying $ h(0)=h(1)=0 $ and $ h(v)>0 $ for $ v\in (0,1) $.
	For every $ \varepsilon>0 $ there exists $ k\in\BB N $ such that
	the $ k $-th Bernstein's polynomial
	\begin{align}\nonumber
	h_k(u)=\sum_{r=1}^{k-1}h\Big(\frac{r}{k}\Big)\binom{k}{r}u^r(1-u)^{k-r}
	\end{align}
	satisfies $ \| h-h_k \|_\infty<\varepsilon $.
	Taking $ \lambda=\sum_{r=1}^{k-1}h(\frac{r}{k}) $ and $ p(i,j)=\lambda^{-1}\mathbbm{1}\{j=i+1\}h(\frac{i}{k}) $, $ 1\le i\le k-1 $,
	we obtain a particle system whose hydrodynamic limit approximates as well as desired the F-KPP equation
	\begin{align}\nonumber
	\partial_t U=\frac{1}{2}\partial_{xx}U-h(U).
	\end{align}
\end{enumerate}
That is, the family of sources $h_p$ produced by this model is dense in the set of continuous functions from $[0,1]$ to $\R_{\ge0}$ that vanishes at the boundary.

\medskip

Once the hydrodynamic limit is established, we follow a strategy previously used in \cite{BCDM+86} to bound from below the asymptotic velocities $v_N$ in terms of the minimal velocity of the limiting equation. The details are given in sections \ref{sctn_hydro} and \ref{sctn_limiting_velocity_kN}. The upper bound is obtained straightforwardly by means of a standard BBM.

\paragraph{{\bf Graphical Construction.}}We end this section with a graphical construction of the $(N,p)$-BBM. Fix $ k\in\BB N $, $ \lambda>0 $ and $p=(p(i,j))_{1\le i< j \le k}$ as in Section \ref{sec_scaling_limit}.
For every $ N\ge k $, we introduce the following three 
elements:
\begin{enumerate}[label=\roman{*}.,font=\itshape]
	\item a random initial configuration $ Y^N_0\in [-\infty,\infty)^N $,
	\item \label{ross}
	an $ N $-dimensional Brownian Motion $ B^N=(B^N_{t}(1),\ldots,B^N_{t}(N))_{t\ge 0} $,
	\item\label{bronte}
	a marked Poisson process  $ (T^N,S^N,R^N)= \{ (T^N_{n},S^N_{n},R^N_n):n\in\BB N\} $.
\end{enumerate}
These random objects are assumed to be defined in the same probability space and for fixed $ N $ are assumed to be independent.
The marks $ T^N_{1}<T^N_{2}<\ldots $ are given by a Poisson point process of intensity $ \lambda N $ in $[0,\infty)$ and represent the jumping times. 
The second coordinates $ S^N_{1},S^N_{2}\ldots $ are $ k $-tuples  of the set of quantiles $ \{ 1,\ldots,N \} $ chosen at random uniformly. 
Finally, $R_n^N$ is a random pair $(i,j)$ with $1\le i<j\le k$ chosen with law $p$.
The $ (N,p) $-BBM $ Y^{N}=\{ Y^{N}_t:t\ge 0 \} $ is constructed as a deterministic function of the triple \textit{i}-\textit{iii}.
Inductively, suppose $ Y^{N} $ has been defined in the time interval $[0, T^N_{n-1}] $ for $ n\ge 1 $ (we use the convention $ T^N_0=0 $), 
and set $ Y^{N}_t = Y^{N}_{T_{n-1}^N}+
B^N_t-B^N_{T_{n-1}^N}$ for $ T^N_{n-1}<t<T^N_{n} $,
and
$ Y^{N}_{T_{n}}=\Gamma_{S^N_{n}(R^N_{n})}(Y^{N}_{T_{n}-}) $. 
Here,
if $S^N_{n}=\{\ell_1,\ldots,\ell_k\}$ with $ \ell_1<\ldots<\ell_k $, then
$S_{n}^N(i,j) = (\ell_{i},\ell_{j})$,
and $ \Gamma_{\ell_{i},\ell_{j}} $ acts on a particle configuration by setting the particle with quantile $ \ell_i $ on top of the one with quantile $ \ell_j $ (a rigorous definition of $ \Gamma $ is given in equation \eqref{def.gamma} below).

\section{Mass-transport Comparison}\label{sctn_comparison}

In this section we consider an extension of the $(b,D)$-BBM, that we call the $ (b,\mathbf{D}) $-BBM.
The difference is that, instead of a sole probability $ D $, the $ (b,\mathbf D) $-BBM is constructed in terms of a sequence of probabilities $ \mathbf D=(D_j)_{j\in\BB N} $.
Also particles are allowed to be located at $ -\infty $.
All the processes appearing in this paper are $ (b,\mathbf{D}) $-BBMs. 
Proposition \ref{thm_comparison}, that gives conditions under which two $ (b,\mathbf{D}) $-BBMs are comparable in the mass-transport sense, allows us to dominate any $(b,D)$-BBM satisfying the hypotheses of Theorem \ref{thm_main} from below and above by treatable processes. 

We start with some basic facts about deterministic particle configurations. For a configuration $  \zeta =(\zeta(1),\ldots,\zeta(N))\in [-\infty,\infty)^N$, we use $\sigma_\zeta$ to denote the permutation on the labels that sorts the particles, using the labels to break ties, i.e. $\sigma_\zeta(i)$ denotes the label of the $i$-th quantile of $\zeta$, and is defined as the only one satisfying the following conditions,
\begin{enumerate}[label=\roman*.]
	\item $ \zeta(\sigma_\zeta(i))\le \zeta(\sigma_\zeta(j))$ if $ i<j $;
	\item $ \sigma_\zeta(i) < \sigma_\zeta(j) $ if $ \zeta(\sigma_\zeta(i))=\zeta(\sigma_\zeta(j)) $ and $ i<j $.
\end{enumerate}
We simplify the notation by writing $ \zeta[i] $ instead of $ \zeta(\sigma_\zeta(i)) $.
For $ 1\le i,j \le N $,
let $\Gamma_{ij}(\zeta) $ be the configuration obtained from $ \zeta $ by putting the particle with quantile $ i $ on top of the one with quantile $j$,
\begin{equation}\label{def.gamma}
\Gamma_{ij}(\zeta)=\eta, \text{ with} \quad \eta(\sigma_\zeta(i))=\zeta[j] \text{ and } \eta(\ell)=\zeta(\ell), \quad \ell\neq \sigma_\zeta(i).
\end{equation}

For 
$ x\in[-\infty,\infty) $, let
$\CAL A_x( \zeta)= (\zeta(1),\ldots,\zeta(N),x)$ be the {\em append} operator.
If $ N\ge 2 $, for $ j\in\{ 1,\ldots,N \} $, let
$ \CAL T_j( \zeta)=(\zeta(1),\ldots,\zeta(\sigma_\zeta(j)-1),\zeta(\sigma_\zeta(j)+1),\ldots,\zeta(N)) \in[-\infty,\infty)^{N-1}$ be the {\em trim} operator that removes the label corresponding to the $ j $-th quantile.

For particle configurations
$ \zeta\in[-\infty,\infty)^N $ and $ \zeta'\in[-\infty,\infty)^{N'} $,
we say that $ \zeta $ is dominated by $ \zeta' $ in the mass-transport sense,
and write $ \zeta\cle \zeta' $,
if
\begin{align*}
\sum_{i=1}^{N}\mathbbm{1}\{ \zeta(i)>x \}\le 
\sum_{i=1}^{N'}\mathbbm{1}\{\zeta'(i)>x \}
\quad
\forall x\in [-\infty,\infty).
\end{align*}
We present the following lemma without proof.

\begin{lemma}\label{domination_properties} 
	Fix $  \zeta\in [-\infty,\infty)^{N}$ and $\zeta'\in[-\infty,\infty)^{N'} $.
	\begin{enumerate}
		\item  \label{pez} The following conditions are equivalent:
		\begin{enumerate}
			\item $  \zeta\cle \zeta' $;
			\item $ N\le  N' $ and
			$ \zeta[i]\le \zeta'[i+ N'-N] $ for every $ i\in\{ 1,\ldots,N \} $;
			\item $ N\le N' $ and
			there exists $ \kappa:\{ 1,\ldots,N \}\to\{ 1,\ldots, N' \} $ injective such that
			$ \zeta(i)\le \zeta'(\kappa (i)) $ for every $ i\in\{ 1,\ldots,N \} $.
		\end{enumerate}
		\item\label{puma}
		$ \zeta\cle  \CAL A_x(\zeta) $ for every $ x\in [-\infty,\infty) $.
		\item\label{liebre}
		For $1\le  i<j \le N $,
		$ \zeta\cle \Gamma_{ij}(\zeta) $.
		\item\label{toro} 
		If $ \zeta\cle\zeta' $, the following properties hold:
		\begin{enumerate}
			\item 
			$ \CAL A_x( \zeta)\cle \CAL A_{x'}(\zeta') $
			for every $-\infty\le x\le x'<\infty $;
			\item 
			$\CAL T_{i}(\zeta)\cle \CAL T_{i+ N'-N}(\zeta') $ for every $ i\in\{ 1,\ldots,N \} $;
			\item 
			if for $ i,j\in\{ 1,\ldots,N \} $
			we call $ i'= i+N'-N $ and $ j'= j+ N'-N $,
			then $ \Gamma_{ij}(\zeta)\cle\Gamma_{ i' j'}(\zeta') $.
		\end{enumerate}
	\end{enumerate}
\end{lemma}

The first statement says that $\zeta \cle \zeta'$ if and only if $\zeta$ can be embedded into $\zeta'$ by a transformation that moves each particle to the right.
Items \ref{puma} and \ref{liebre} mean that
the particle configuration increases if a particle is added or if a particle jumps to the right.
Item \ref{toro} says that the order is preserved if we add a particle, if we remove one particle, or if a particle jumps on top of another one, provided the involved particles are properly chosen.

Unlike the $(b,D)$-BBM, the $(b,\mathbf{D})$-BBM that we define now allows the killing probability to depend on the quantile of the branching particle. 
Between jumping times,  particles move as independent Brownian Motions,
and the quantile $ j $ of the branching particle is determined in terms of $ b $ as before.
If $ j=1 $, the quantile  of the particle that is going to be killed is chosen to be $ i=-\infty $ (no killing). If otherwise $ j>1 $,
we have $ i=-\infty $ with probability $ D_{j-1}(-\infty) $ and 
for $ 1\le i' < j $, $i=i'$ with probability
$D_{j-1}\big(\frac{i'}{j-1}- \big) - D_{j-1}\big(\frac{i'-1}{j-1}-\big)$.
Of course the $ (b,D) $-BBM is a $ (b,\mathbf{D}) $-BBM with $ D_j=D $ for every $ j\in\BB N $.

Let $ X $ and $ {X}^{'} $ be a $ (b,\mathbf{D}) $-BBM and a $ (b',\mathbf{D}') $-BBM respectively. We omit writing the superscripts indicating the initial number of particles when no confusion can arise. We say that (the initial condition) $ X_0$ is stochastically dominated by $  X_0^{'} $, and write $ X_0\le_{st}  X_0^{'} $, if they can be coupled in such a way that $ X_0\cle X_0^{'} $ almost surely.
We say that the process $ X $ is stochastically dominated by $ X^{'} $,  and write $ X\le_{st} X^{'} $, if they can be coupled in such a way that almost surely $X_t\cle  X_t^{'}$ for every $t\ge 0$.

\begin{proposition}\label{thm_comparison}
	Suppose that $ b $ and $ b' $ satisfy
	$ b(x)\le  b'(x') $ if $ x\le x' $.
	Suppose further that $ \mathbf{D}=(D_j)_{j\in\BB N} $ and $ \mathbf{D}'=( D'_j)_{j\in\BB N} $ are such that, for every $ j,j'\in\BB N $, $ {D_j}\le { D'_{j'}} $ pointwise.
	Let $ X $ and $X^{'} $ be a $ (b,\mathbf{D}) $-BBM and a $ (b',\mathbf{D}') $-BBM respectively with
	random initial conditions satisfying
	$ X_0\le_{st} X^{'}_0 $.
	Then $ X\le_{st} X^{'} $.
\end{proposition}

Observe that the hypotheses over $ b $ and $ b' $ are satisfied if any of the following two conditions hold:
\begin{enumerate}
	\item[i.] $ b $ is non-decreasing and $ b\le  b' $ pointwise.
	\item[ii.] $ \sup_{x\in [0,1]}b(x)\le \inf_{x\in [0,1]} b'(x) $.
\end{enumerate}

\begin{proof}
	Let $N, N'$ be total number of particles in $X_0, X_0'$ respectively. Lemma \ref{domination_properties} implies $ N \le N' $.
	For every $ j\in\{ 1,\ldots,N \} $, call $ j'= j+ N'-N $.
	Consider exponential random variables $\{ \mathsf W_{\ell}:\ell\in\{ 1,\ldots,N \}\} $ and $\{{\mathsf W}'_{\ell}:\ell\in\{ 1,\ldots, N' \}\} $ such that $\mathsf W_\ell $ has rate $ b(\frac{\ell-1}{N-1}) $ for every $ \ell\in\{1,\ldots,N\} $ and 
	$ {\mathsf W}'_\ell $ has rate $  b'(\frac{\ell-1}{ N'-1}) $ for every $ \ell\in\{1,\ldots, N'\} $.
	Since $ b(\frac{j-1}{N-1})\le b'(\frac{ j'-1}{ N'-1}) $ for every $ j\in\{ 1,\ldots,N \} $, we can  couple them in such a way that, for every $ j\in\{ 1,\ldots,N \} $,   $ \mathsf W_{j}\ge {\mathsf W}'_{j'} $. To be precise, let $\mathsf X_1,\ldots, \mathsf X_{N'}$ be  independent two-dimensional Poisson point processes of intensity 1, and define $ \mathsf W'_{\ell}=\inf\{z>0:\big([0,z]\times [0,b'(\frac{\ell-1}{N'-1})]\big)\cap \mathsf {X}_\ell\neq\emptyset \}  $ for $ 1\le \ell\le N'-N  $,
	and 
	$ {\mathsf W}_{j}=\inf\{z>0:\big([0,z]\times [0, b(\frac{j-1}{N-1})]\big)\cap \mathsf  {X}_{j'}\neq\emptyset \}  $ 
	and 
	$ \mathsf W'_{j'}=\inf\{z>0:\big([0,z]\times [0,b'(\frac{j'-1}{N'-1})]\big)\cap \mathsf {X}_{j'}\neq\emptyset \}  $
	for $ j\in\{1,\ldots,N\} $.
	Call $ \tau' =\min_{\ell\in\{1,\ldots,N'\}} {\mathsf W} '_{\ell} $.
	In the time interval $ [0,\tau' ) $,
	we couple the Brownian displacements in such a way that 
	$ X_t(\sigma_{X_0}(j))- X'_t(\sigma_{X'_0}(j')) =X_0[j]-  X'_0[j'] $ for every $ j\in\{ 1,\ldots,N \} $ and every $ t\in [0,\tau' ) $ (we are coupling the trajectories of the $ N $ labels that are at the rightmost positions at time $ t=0 $).
	Item (\ref{pez}) of Lemma \ref{domination_properties} readily implies $ X_{t}\cle X'_{t} $ for every $ t\in [0,\tau') $.
	
	The particle configurations
	$ X_{\tau'} $ and $X'_{\tau'} $ will be constructed from a case-dependent modification of $ X_{\tau'-} $ and $  X'_{\tau'-} $.
	By an iterative argument, we can conclude once we have proven that $ X_{\tau'} \cle  X'_{\tau'} $.
	Let $l'= \argmin\{{\mathsf W} '_{\ell}:\ell\in\{ 1,\ldots, N' \}\}$. We split into cases:
	\begin{enumerate}
		\item If $ l'\le N'-N $,
		we set $ X_{\tau'}=X_{\tau'-} $, and use $D'_{l'-1} $ to obtain $ X'_{\tau'} $ from $X'_{\tau' -} $.
		Items \ref{puma} and \ref{liebre} of Lemma \ref{domination_properties}  guarantee that  $ X'_{\tau'- }\cle  X'_{\tau'} $, implying
		$ X_{\tau' }\cle  X'_{\tau' } $.
		\item 
		If $ l'>  N'-N$, we split again into two subcases:
		\begin{itemize}
			\item[(i)] If $ \mathsf W_{l} > \mathsf W'_{l'}$ ($ l= l'-(N'-N) $) we proceed as before. We set $ X_{\tau'}=X_{\tau'-} $, and use $  D'_{ l'-1} $ to modify $ X'_{\tau'-} $ and obtain $ X'_{\tau'} $ with $X_{\tau'}\cle  X'_{\tau'} $.
			\item[(ii)] If $ \mathsf W_{l} = \mathsf W'_{l'}$, call $ \eta= X_{\tau'-} $ and $ \eta'=   X'_{\tau'-} $. Let $ \xi \in[-\infty,\infty)^{l-1}$ (resp. $  \xi' \in[-\infty,\infty)^{ l'-1}$) be the particle configuration obtained from $ \eta $ (resp. from $ \eta' $) after removing the $ N-(l-1) $($ = N'-( l'-1) $) right-most particles. We are removing the particles $ \eta[l],\ldots,\eta[N] $ (resp. $  \eta'[ l'],\ldots, \eta'[ N'] $).
			We proceed to couple the quantiles of the particles that are going to be killed. For a distribution function $D$ on $ [-\infty,\infty) $,  consider the generalized inverse $D^{-1}:[0,1]\to [-\infty,\infty) $ defined by
			\begin{align}\nonumber
			D^{-1}(y)= \inf\{x\in\BB R: D(x)\ge y\}.
			\end{align}
			If $\mathsf U $ is a random variable uniformly distributed in $[0,1]$, then the (extended) random variable $ D^{-1}(\mathsf U) $ has law $ D $.
			The quantiles $ m $ and $ m' $ are defined by
			\begin{align*}\nonumber
			m = & -\infty \cdot  \mathbbm{1}\{D_{l-1}^{-1}(\mathsf U)=-\infty\}\\ &+
			\sum_{i=1}^{l-1}i \cdot \mathbbm{1}\big\{
			\tfrac{i-1}{l-1}\le  D_{l-1}^{-1}(\mathsf U)< \tfrac{i}{l-1} \big\}
			\\ \nonumber
			m'= & -\infty \cdot  \mathbbm{1}\{ (D'_{ l'-1})^{-1}(\mathsf U)=-\infty\} \\ & + 
			\sum_{i=1}^{ l'-1}i \cdot \mathbbm{1}\big\{
			\tfrac{i-1}{ l'-1}\le   (D'_{ l'-1})^{-1}(\mathsf U)< \tfrac{i}{l'-1} \big\},
			\end{align*}
			with the convention $ -\infty \cdot 0=0 $.
			Next we prove that $   m'\le m+  N'-N $.
			If $ m=-\infty $ then $ D^{-1}_{l-1}(\mathsf U)=-\infty $, that implies $ (D'_{ l'-1})^{-1}(\mathsf U)=-\infty $ since $  D_{l-1} \le  D'_{ l'-1} $ pointwise. So $  m'=-\infty $ and the desired inequality holds.
			If $ m\neq -\infty$,
			we have
			\begin{align*}
			D_{l-1}^{-1}(\mathsf U)<\frac{m}{l-1}
			\le
			\frac{m+( l'-l)}{l-1+( l'-l)}
			=
			\frac{m+ l'-l}{ l'-1},
			\end{align*}
			implying $  (D'_{ l'-1})^{-1}(\mathsf U)< \frac{m+ l'-l}{l'-1} $ (again because $  D_{l-1} \le   D'_{ l'-1} $ pointwise). So $  m'\le m+ l'-l=m+ N'-N $.
			Let $ \theta $ and $ \theta' $
			be the particle configurations obtained respectively from $ \xi $ and $  \xi' $ after 
			removing the quantiles $ m $ and $ m'$.
			Item \eqref{toro} in Lemma \ref{domination_properties} implies the dominance
			$ \theta\cle \theta' $.
			Finally, let $ \gamma $ (resp. $ \gamma' $) be the configuration obtained from $ \theta $ (resp. $ \theta' $)
			after (\textit{a}) adding the $ N-(l-1) $ particles that have been removed in the transformation
			from $ \eta $ to $ \xi $ (resp. from $  \eta' $ to $  \xi' $),
			and (\textit{b}) adding an extra particle at position $ \eta[l] $ (resp. $  \eta'[ l'] $).
			Again item \eqref{toro} in Lemma \ref{domination_properties} implies $ \gamma\cle \gamma' $.
			This subcase follows because $ \gamma=X_{\tau'} $ and $ \gamma'=  X'_{\tau'} $.
		\end{itemize}
	\end{enumerate}
	The proof is now complete. \qedhere
	
\end{proof}

\paragraph{{\bf A Lower Bound}} We end this section showing that under minimal assumptions on $b$ and $D$, the $(b,D)$-BBM can be bounded from below by an $(N,p)$-BBM with an adequately chosen $p$.

\begin{proposition}\label{bailey}
	Assume $D(x)\ge x^{k-1}$ and $b(x) \ge x^{k-1}$ for all $0\le x \le 1$. Let $X^N $ be a $(b,D)$-BBM 
	and $ Y^{N} $ an $ (N,p) $-BBM with $p(i,j)=\mathbbm{1} \{i=k-1,j=k\}$ and $\lambda = 1/k$.
	If $ Y^{N}_0\le_{st} X^N_0 $
	then $ Y^{N}\le_{st} X^N $.
\end{proposition}

\begin{proof}
	The key observation is that
	$ Y^{N} $ can be thought as a $ (\hat b,\hat{\mathbf{D}}) $-BBM with the proper choice of $ \hat b $ and $ \hat{\mathbf{D}} $. Observe that in the $(N,p)$-BBM, all the quantiles with $ j\le k-1 $ do not branch.
	If $ j\ge k $, in order to have a branch at quantile $j$ we need to choose a $k$-tuple such that $j$ is the largest quantile in it. Hence its branching rate is given by
	\begin{align*}
	\frac{N}{k}\,\frac{\binom{j-1}{k-1}}{\binom{N}{k}}=
	\frac{j-1}{N-1} \, \frac{j-2}{N-2} \, \ldots \,
	\frac{j-(k-1)}{N-(k-1)}
	\eqqcolon \lambda_{j}.
	\end{align*}
	
	Given the event that quantile $ j\ge k $ has a branch, the probability of a quantile smaller or equal than $i \in \{k-1,\dots,j-1\}$ being killed is
	\begin{align*}
	\frac{\binom{i}{k-1}}{\binom{j-1}{k-1}}
	=\frac{i}{j-1}\,
	\frac{i-1}{j-2}\,
	\ldots
	\, \frac{i-(k-2)}{j-(k-1)}
	\eqqcolon q_{ij},
	\end{align*}
	and zero if $ i<k-1 $.
	Define
	\begin{align*}
	\hat b(x)=
	\lambda_N \mathbbm{1}\{ x=1 \}+\sum_{j=k}^{N-1} \lambda_j \mathbbm{1}\{ \tfrac{j-1}{N-1}\le x<\tfrac{j}{N-1} \}
	\end{align*}
	and
	\begin{align*}
	{\hat D_{j-1}}(x)=\mathbbm{1}\{x\ge 1 \}+\sum_{i=k-1}^{j-2} q_{ij} \mathbbm{1}\{ \tfrac{i}{j-1}\le x<\tfrac{i+1}{j-1} \}.
	\end{align*}
	Under these definitions, it is easy to check that $ Y^{N} $ is a $ (\hat b,\hat{\mathbf{D}}) $-BBM with $\hat{\mathbf{D}} = (\hat D_j)_{j\in\BB N}$. Since
	$\lambda_j\le (  \frac{j-1}{N-1} ) ^{k-1}$ and $ q_{ij}\le (\frac{i}{j-1})^{k-1}$, we get $ {\hat D_i}(x) \le x^{k-1} \le D(x)$ for $0\le x \le 1$,
	and $\hat b(\hat x) \le \hat x^{k-1} \le x^{k-1} \le b(x)$ for $ 0\le \hat x\le x\le 1 $. We can apply Proposition \ref{thm_comparison} to conclude.
\end{proof}

\section{Hydrodynamics for the \texorpdfstring{$ (N,p) $}{(N,p)}-BBM}\label{sctn_hydro}
We now prove Theorem \ref{thm_hydro}. We stress that the time $ t>0 $ is fixed during all the proof.
In this section we use $ F^N $ for,
\begin{align*}
F^N(t,x)= \frac{1}{N}\sum_{i=1}^{N} \mathbbm{1}\{Y^N_t(i) \le x\}.
\end{align*}
Since $ U(t,\cdot) $ is continuous, convergence \eqref{reposera} is equivalent to 
\begin{align*}
\lim_{N\to\infty}  F^N(t,x)=U(t,x) \qquad\mbox{in probability}
\end{align*}
for every $ x\in \BB R $. For $ \varepsilon>0 $, we have
\begin{align}\label{carmen}
\BB P (|F^N(t,x)-U(t,x) |>\varepsilon)
=
\int \nu^N(\mathrm{d}\zeta)
\BB P_\zeta(|  F^N(t,x)-U(t,x) |>\varepsilon),
\end{align}
where $ \nu^N $ is the distribution of $ Y^N_0 $ on $ [-\infty,\infty)^N $
and $ \BB P_\zeta(\cdot)= \BB P(\cdot|Y^N_0=\zeta)$.
Recall the definition of $ h_p $ given in \eqref{defi0002}. 
For every $ \zeta\in [-\infty,\infty)^N $,
let $ U_{\zeta} $ be the unique bounded solution to
\begin{align}\nonumber
&\partial_t U_{\zeta}=\tfrac{1}{2}\partial_{xx}U_{\zeta}- h_p(U_{\zeta}),
\\
\nonumber
&U_{\zeta}(0,\cdot)=F_\zeta,
\end{align}
being $ F_\zeta $ the distribution function of the empirical probability of $ \zeta $.
Splitting into the cases $ |U_\zeta(t,x)-U(t,x)|>\tfrac{\varepsilon}{2} $
and $ |U_\zeta(t,x)-U(t,x)|\le\tfrac{\varepsilon}{2} $,
\eqref{carmen} can be bounded from above by
\begin{align}\label{astro}
\nonumber 
&\int \nu^N(\mathrm{d} \zeta)
\BB P_\zeta(|F^N(t,x)  -U_\zeta(t,x)|>\tfrac{\varepsilon}{2})\\
&\quad+
\int \nu^N(\mathrm{d} \zeta)\mathbbm{1}\{|U_{\zeta}(t,x)-U(t,x)|>\tfrac{\varepsilon}{2}\}.
\end{align}
The second term in \eqref{astro}
vanishes in the limit as $ N\to\infty $ due to our assumptions and Theorem \ref{thm_quinoto}.
For $ (s,y)\in [0,\infty)\times\BB R $, let $ U^N_\zeta(s,y)=\BB E_\zeta(F^N(s,y)) $, being $ \BB E_\zeta $ the expectation with respect to $ \BB P_\zeta $.
Splitting  into the cases 
$ |F^N(t,x)-U^N_\zeta(t,x)|>\frac{\varepsilon}{4} $
and
$ |F^N(t,x)-U^N_\zeta(t,x)|\le \frac{\varepsilon}{4} $,
and using Tchebyshev's inequality,
the first term in \eqref{astro} can be bounded by
\begin{align}\label{caberna}
\nonumber &\tfrac{16}{\varepsilon^2}\int\nu^N(\mathrm{d}\zeta)\,
[\BB E_\zeta(F^N(t,x)^2)  -U^N_\zeta(t,x)^2]\\
&\quad+
\int\nu^N(\mathrm{d}\zeta)
\mathbbm{1} \{| U_\zeta^N(t,x)-U_\zeta(t,x) | > \tfrac{\varepsilon}{4}\}.
\end{align}
The first term in this expression vanishes  in the limit as $ N\to\infty $ due to the next result and the dominated convergence theorem.

\begin{lemma}[Propagation of Chaos]\label{caramelo}
	For every $ t\ge 0 $ and $ \ell\in\BB N $ there is a constant $C>0$ such that,
	\begin{align}\nn
	\sup_{\zeta\in [-\infty,\infty)^N}
	\, \sup_{(s,x)\in [0, t]\times\BB R}
	|\BB E_\zeta(F^N(s,x)^\ell)-U^N_\zeta(s,x)^\ell|
	\le \frac{C}{N}.
	\end{align}
\end{lemma}

We prove this result in Appendix \ref{proof_correlations}.
We now turn to control the second term in \eqref{caberna}.
Recall the definition of $ \hat p(r) $ given in \eqref{p_hat}.

\begin{lemma}\label{thm_finite_hydro}
	Let $ h^N_{p}:[0,1]\to \BB R $ be the function defined by
	\begin{align*}
	h^N_{p}(u)=
	\lambda \sum_{r=1}^{k-1} \hat p(r)\binom{k}{r}
	w^N_{r}(u),
	\end{align*}
	with
	\begin{align*}
	w^N_{r}(u)
	=
	\left[ \prod_{\ell=0}^{r-1}  (u-\tfrac{\ell}{N}) \right]
	\left[ \prod_{\ell=0}^{(k-r)-1}  (1-u-\tfrac{\ell}{N}) \right]
	\left[\prod_{\ell=0}^{k-1}  \tfrac{N}{N-\ell}\right].
	\end{align*}
	Then, for every $ \zeta\in [-\infty,\infty)^N $,
	$ U^N_{\zeta}$ verifies
	\begin{align}\label{kiwi}
	&\partial_t U^N_{\zeta} =\tfrac{1}{2}\partial_{xx} U^N_{\zeta}- \BB E_\zeta[ h^N_p(F^N)],
	\\
	\nonumber	&U^N_{\zeta}(0,\cdot)=F_\zeta.
	\end{align}
\end{lemma}
Before proving it, we show how to conclude.
Equation \eqref{kiwi} can be written as
\begin{align*}
\partial_t U_\zeta^N
=\tfrac{1}{2}\partial_{xx}U_\zeta^N
- h_p(U_\zeta^N)
+ \CAL E^N_{1,\zeta}
+ \CAL E^N_{2,\zeta}
\end{align*}
with the errors defined as
\begin{align*}
&\CAL E^N_{1,\zeta}
=h_p(U^N_\zeta)-\BB E_\zeta(h_p(F^N))
\\ \nn
&\CAL E^N_{2,\zeta}=
\BB E_\zeta(h_p(F^N))-\BB E_\zeta(h^N_p(F^N)).
\end{align*}
The comparison principle Theorem \ref{thm_quinoto} allows us to conclude once we prove that
\begin{align*}
\lim_{N\to\infty}\sup_{(s,y)\in [0,t]\times\BB R}| \CAL E^N_{\ell,\zeta}(s,y)|=0,
\quad \ell=1,2.
\end{align*}
The case $ \ell=1 $ follows by Lemma \ref{caramelo}, while for $ \ell=2 $ the limit holds since
\begin{align*}
\lim_{N\to\infty} \|h^N_p-h_p\|_\infty =0.
\end{align*}
To see why this is true observe that it is enough to prove uniform convergence of each  $ w_{r}^N $, 
namely
\begin{align}
\lim_{N\to\infty}\sup_{u\in [0,1]}|w_r^N(u)-u^r(1-u)^{k-r}|,
\end{align}
that follows because the first $ r $ factors defining $ w_r^N $ uniformly converge to $ u\mapsto u $, the second $ (k-r) $ to $ u\mapsto 1-u $, and the last $ k $ to $ u\mapsto 1 $.
Heuristically, we are approximating the sampling of $ k $ particles without replacement $h_p^N$,  by sampling with replacement $h_p$, which certainly holds in the limit $N\to \infty$.

\begin{proof}[Proof of Lemma \ref{thm_finite_hydro}]
	\label{zebra}
	Fix $ s\in (0,t] $ and, for  $ \ell\in\{ 1,\ldots,N \} $, call 
	\begin{align}\nonumber
	q_{\ell}(s,x)= \BB P_\zeta[Y^N_{s}(\ell)\le x].
	\end{align}
	We consider the cases $  T^N_1 > s   $ and $ T^N_1 \le s  $ (recall the graphical construction of Section \ref{sec_scaling_limit}) to get
	\begin{align}\label{yayo}
	\nonumber q_{\ell}(s,x)
	=&
	\, e^{-\lambda N s}\int_{-\infty}^\infty
	\Phi(s, x-y)  \mathbbm{1}\{ y\ge \zeta(\ell) \}\mathrm{d} y
	\\ & +\BB E_\zeta\big[\mathbbm{1}\{ Y^N_{s}(\ell)\le x \}\mathbbm{1}\{T^N_1\le s\}\big].
	\end{align}
	Here $ \Phi $ is the Gaussian kernel
	\begin{align}\label{cuadrito}
	\Phi(s,z)=\frac{1}{\sqrt{2\pi s}}e^{-\frac{z^2}{2s}}.
	\end{align}
	Let $ T^N_* $ be the last jump before $ s $. We use total probability conditioning first on $ T^N_*$, that  has densisty $ \lambda N e^{-\lambda N(s-r)}\mathbbm{1}\{ 0<r<s\} $ on the event $ T_1^N\le s $
	(the law of $ s-T_*^N $ on $ T_1^N\le s $ is exponential with rate $ \lambda N $ truncated to $ (0,s) $), and then on $x-  [B^N_s(\ell)-B^N_{T^N_*}(\ell)] $, whose law conditioned to $T^N_*=r$ is $ \Phi(s- r ,x-y)\mathrm{d} y $. The second term in the right-hand side of\eqref{yayo} can be written as
	\begin{align}\label{alacran}
	\int_0^s \lambda N e^{-\lambda N (s-r) } {\mathsf g}(r)\mathrm{d}  r,
	\end{align}
	where $ \mathsf g(r) $ is defined by formula
	\[
	\int_{-\infty}^\infty
	\Phi(s- r ,x-y)\BB P_\zeta(  Y^N_{s}(\ell)\le x  | T^N_*= r ,B^N_s(\ell)-B^N_r(\ell)=x-y ) \mathrm{d} y.
	\]
	
	Observe that
	\begin{align*}
	&\BB P_\zeta (  Y^N_{s}(\ell)\le x  | T^N_*= r ,B_s^N(\ell)-B_r^N(\ell)=x-y )
	\\
	&\quad =
	\int \mathrm{d} S \sum_{1\le i<j\le k} p(i,j) \,\BB P_\zeta [ \Gamma_{S(i,j)}(  Y^N_{ r })(\ell)\le y ]=:g_\ell( r ,y) ,
	\end{align*}
	where $ \mathrm{d} S $ is the law of a $ k $-tuple uniformly chosen at random (for a $ k $-tuple $ S $, recall the definition of $ S(i,j) $ given at the end of Section \ref{sec_scaling_limit}:
	if $S=\{\ell_1,\ldots,\ell_k\}$ with $ \ell_1<\ldots<\ell_k $, then
	$S(i,j) = (\ell_{i},\ell_{j})$).
	Plugging-in \eqref{alacran}, we obtain 
	\begin{align}\nonumber
	q_\ell(s,x)=&
	\int_{-\infty}^\infty G(s,x-y)  \mathbbm{1}\{ y\ge \zeta(\ell) \}\mathrm{d} y\\
	\nonumber &  +\int_0^s \int_{-\infty}^\infty G(s- r ,x-y)\lambda N g_\ell( r ,y) \mathrm{d} y\mathrm{d}  r ,
	\end{align}
	where $ G(r,z)= e^{-\lambda N r}\Phi(r,z) $
	is the Green kernel associated to equation
	\begin{align}\nonumber
	\partial_t V=\tfrac{1}{2}\partial_{xx}V-\lambda N V.
	\end{align}
	Since $ s\in (0,t] $ is arbitrary, we conclude (see Appendix \ref{appendix1}) that $ q_\ell $ solves
	\begin{align}\nonumber
	&\partial_tq_\ell=\tfrac{1}{2}\partial_{xx}q_\ell-\lambda N (q_\ell- g_\ell),
	\\
	\nonumber
	&q_\ell(0,x)=\mathbbm{1}\{ x\ge \zeta(\ell) \}.
	\end{align}
	Summing over $ \ell\in\{ 1,\ldots,N \} $ and dividing by $ N $,
	we get
	\begin{align}\label{traviata}
	&\partial_t U^N_{\zeta} =\tfrac{1}{2}\partial_{xx} U^N_{\zeta}
	-\lambda N\Big(U^N_{\zeta}
	-\frac{1}{N}\sum_{\ell=1}^N g_\ell\Big),
	\\
	\nonumber &U^N_{\zeta}(0,\cdot)=F_\zeta.
	\end{align}
	Observe that
	\begin{align*}
	\frac{1}{N}\sum_{\ell=1}^N
	g_\ell(s,x)=
	\int \mathrm{d} S \sum_{1\le i <j\le k} p(i,j) \, \BB E_\zeta[ F_{\Gamma_{S(i,j)}(Y^N_s)}(x) ].
	\end{align*}
	For fixed $ S $, $ (i,j) $ and writing $F_{ij}(x)$ for $F_{\Gamma_{S(i,j)}(Y^N_s)}(x)$, we have
	\begin{align*}
	\BB E_\zeta[ F_{ij}(x) ]
	=\sum_{m=0}^N\BB E_\zeta [F_{ij}(x)|
	F^N(s,x)
	=\tfrac{m}{N} ]
	\,\BB P_\zeta[F^N(s,x)= \tfrac{m}{N} ].
	\end{align*}
	On the event $ F^N(s,x)=\frac{m}{N} $,
	we have $F_{ij}(x)=\frac{m-1}{N} $  if a particle jumps over $ x $, and $F_{ij}(x)=\frac{m}{N} $ otherwise.
	Then
	\begin{align*}\nonumber
	&\int \mathrm{d} S \sum_{1\le i < j \le k} p(i,j) \,
	\BB E_\zeta[F_{ \Gamma_{S(i,j)}(Y^N_{s}) }(x)|
	F^N(s,x)= \tfrac{m}{N} ]
	\\
	&\quad =\frac{m-1}{N}p_{N,m}+\frac{m}{N}(1-p_{N,m})=\frac{m}{N}-\frac{p_{N,m}}{N},
	\end{align*}
	where
	\begin{align*}\nonumber
	p_{N,m}= \sum_{r=0}^k \frac{\binom{m}{r}\binom{N-m}{k-r}}{\binom{N}{k}}
	\,\hat p(r)
	\end{align*}
	is the probability of such a jump ($ \binom{a}{b} $ is assumed to be zero for $ a<b $). Then
	\begin{align}\nonumber
	\frac{1}{N}\sum_{\ell=1}^N
	g_\ell(s,x)
	&=\sum_{m=0}^N\frac{m}{N}\BB P_\zeta [F^N(s,x)= \tfrac{m}{N} ]
	-\sum_{m=0}^N \frac{p_{N,m}}{N}\BB P_\zeta[F^N(s,x)= \tfrac{m}{N} ]
	\\ \nonumber
	&= U^N_{\zeta}(s,x)-\sum_{m=0}^N \frac{p_{N,m}}{N}\BB P_\zeta[F^N(s,x)= \tfrac{m}{N} ].
	\end{align}
	Then the second term in the right-hand side of \eqref{traviata} can be written as
	\begin{align*}
	\lambda\sum_{m=0}^N p_{N,m} \BB P_\zeta[F^N(s,x)=\tfrac{m}{N}].
	\end{align*}
	We can conclude if we prove that $ \lambda p_{N,m} = h_p^N(\tfrac{m}{N}) $.
	After a trivial comparison of terms in the corresponding sums,
	this follows from
	\begin{align}\label{polemic}
	\frac{\binom{m}{r}\binom{N-m}{k-r}}{\binom{N}{k}}
	=\binom{k}{r}w_r^N\Big(\frac{m}{N}\Big)
	\end{align}
	for every $ r $. This equality holds since
	\begin{align}
	\binom{k}{r}w_r^N\Big(\frac{m}{N}\Big) = \binom{k}{r}\frac{\frac{m!}{(m-r)!}\frac{(N-m)!}{((N-m)-(k-r))!}}{\frac{N!}{(N-k)!}} = \frac{\frac{m!}{(m-r)!r!}\frac{(N-m)!}{((N-m)-(k-r))!(k-r)!}}{\frac{N!}{(N-k)!k!} }  = \frac{\binom{m}{r}\binom{N-m}{k-r}}{\binom{N}{k}}.
	\end{align}
\end{proof}

\section{Limiting velocity for \texorpdfstring{$ (N,p) $}{(N,p)}-BBM}\label{sctn_limiting_velocity_kN}

Take $ k\ge 2 $, $ \lambda>0 $ and $ p=(p(i,j))_{1\le i<j\le k} $ as before,
and let 
\begin{align*}
i_0=\min\{i:\mbox{$ p(i,j)>0 $ for some $ j $} \}-1.
\end{align*}
The role of $ i_0 $ can be informally explained as follows. Consider for a moment
the $ (N,p) $-BBM with no particle located at $ -\infty $. At a jump time, the $ i_0 $ leftmost particles do not jump due to the definition of $ i_0 $. Suppose that, at some point, the particles with quantile no greater than $ i_0+1 $ are far from the cloud of particles. From this time on the particles to the left of the $i_0$ quantile are not likely to perform any jump and hence they can be thought as driftless. On the other hand, the cloud of particles to the right of the $i_0$ quantile has a positive drift induced by the jumps. So, these two clouds are likely to separate from each other as time goes by. For this reason we use below the technical notational trick of assuming that the initial configuration has $ i_0 $ particles located at $ -\infty $ (and hence they remain there forever).

We will construct an auxiliary Markov process 
$ Z^{N}=\{ Z^{N}_t:t\ge 0\} $
with state-space $ \BB R^{N-i_0} $
as a function of the $ (N,p) $-BBM and the initial condition $ Z_0^{N} $. Consider 
the $ (N,p) $-BBM $ Y^N $ with initial condition
\begin{align*}
Y_0^N(i)=
\begin{cases}
-\infty  & \mbox{if $ 1\le i\le i_0 $}
\\
Z_0^N(i-i_0)& \mbox{if $ i_0< i\le N $}
\end{cases},
\end{align*}
and set
\begin{align*}
Z^{N}_t=(Y_t^N(i_0+1),\ldots,Y_t^N(N))
\end{align*}
for $ t>0 $.
That is, $ Z^{N} $ is the projection of $ Y^N $ over the $ N-i_0 $ right-most particles, those that play a meaningful role.
The Markovian property of $ Z^{N} $ follows because, as mentioned before, in the process $ Y^N $, particles with label smaller or equal than $ i_0 $ in $ Y^N $ remain at $ -\infty $ for every time. 
Observe that if there exists $ j $ for which $ p(1,j)>0 $ then
$ Z^{N} $ is simply $ Y^N $.
For fixed $ N $, the process $ Z^N $ has a well defined velocity:

\begin{proposition}\label{caramba}
	For every $ N $, there exists $ w_{N}\in \BB R $ such that,
	for every random initial distribution $ Z_0^{N} \in\BB R^{N-i_0}$, 
	the limits
	\begin{align*}
	\lim_{t\to\infty}\frac{1}{t}Z^{N}_t[1]=
	\lim_{t\to\infty}\frac{1}{t}Z^N_t[N-i_0]=
	\lim_{t\to\infty}\frac{1}{t}Y^N_t[N]=
	w_{N}
	\end{align*}
	hold a.s. and in $ L^1 $.
\end{proposition}
\begin{proof}
	We omit the details due to the similarity with the proofs of Proposition 2 and Theorem 2(a) in \cite{BG10} and \cite{DR}  respectively.
	It is a consequence of Liggett's subadditive ergodic theorem.
	The key requirement is that, if we run the process until the $ m $-th jumping time $ T_m $,
	restart it with the $ N-i_0 $ particles at the position of the rightmost one,
	and run it until we have another $ n $ extra jumps, then the resulting configuration dominates the configuration we would get by running the process until the $ (m+n)$-th jumping time.
	This requirement if fulfilled by Proposition \ref{thm_comparison} and the fact that the $(N,p)$-BBM is a $ (b,\mathbf{D}) $-BBM.
\end{proof}

We now prove the lower bound for the velocities.
\begin{proposition}\label{uttanasana}
	The limiting velocity $ w_N $ of the right-most particle of the $(N,p)$-BBM satisfies
	\begin{align}\label{mora}
	\liminf_{N\to\infty}w_{N} \ge c^*.
	\end{align}
	Here $ c^*>0 $ is the minimal velocity of equation \eqref{kpp_eq_1}, see Appendix \ref{appendix1}.
\end{proposition}

\begin{remark}
	Recall the definition of $ h_p $ given in \eqref{defi0002}.
	It is well known that $ c^*\ge \sqrt{-2h_p'(1)} $.
	Since $-h_p'(1) = \lambda k \hat p(k-1)$, 
	where $\hat p(k-1)$ is the probability of having a particle jumping on top of the rightmost one in th $k-$tuple at a branching-selection event, this bound has a natural interpretation in terms of the parameters of the model.
	Observe that it is not sharp in many cases,
	for example when $\hat p(k-1)=0$. But it is good enough in several situations as we will see.
\end{remark}

To prove Proposition \ref{uttanasana} we follow a strategy recently used in \cite{GJM}.
Let $  \hat Z^{N}=\{\hat Z^{N}_t:t\ge 0 \} $ be the process $ Z^{N} $ as seen from its leftmost particle,
\begin{align*}
\hat Z^{N}_t(i)
=Z^{N}_t[i+1]
-Z^{N}_t[1], \quad i\in
\{1,\ldots,N-i_0-1 \}, t\ge 0.
\end{align*}
We will make use of its stationary distribution.
\begin{proposition}\label{kilo}
	The process $ \hat Z^{N} $ has a unique stationary distribution $ \hat\nu^{N} $.
\end{proposition}

\begin{proof}
	The result follows by showing that the process is Harris recurrent.
	The proof is very similar to the one of \cite[Theorem 2.3]{GJM} for the special case $ k=2 $, so we omit it.
\end{proof}

Let $ \nu^{N} :=\delta_0\otimes \hat \nu^N$,
namely a particle is fixed at the origin and the remaining $ N-1 $ are sampled according to $ \hat\nu^N $,
and for $ t\ge 0 $ let
\begin{align*}
M^{N}_t:=\frac{1}{N-i_0}\sum^{N-i_0}_{i=1}Z_t^{N}(i)
\end{align*}
be the empirical mean of $ Z^{N}_t $.
The following result gives a formula for the velocity in term of the empirical mean.
The proof can be borrowed with no modification from the
one given below formula (4.9) in \cite{GJM}.
\begin{proposition}\label{panic}
	For every $ t> 0 $,
	\begin{align*}
	w_{N}= \frac{\mathrm{d}}{\mathrm{d} t}\,\BB E_{\nu^{N}}[M^{N}_t].
	\end{align*}
\end{proposition}


Since for every particle configuration $\zeta\in\BB R^{N-i_0}$ we have
\begin{align}\nn
\frac{1}{N-i_0}\sum^{N-i_0}_{i=1}\zeta(i)
=
\int_{0}^\infty 1- F_\zeta(x)\mathrm{d} x-\int_{-\infty}^0 F_\zeta(x)\mathrm{d} x,
\end{align}
if we call $ G^{N} $ the distribution function of the empirical law of $ Z^{N} $,
we have
\begin{align}\nn
\BB E_{\nu^{N}}[M^{N}_t]
=
\int_{0}^\infty \BB E_{\nu^{N}}[ 1- G^{N}(t,x)]\mathrm{d} x-\int_{-\infty}^0 \BB E_{\nu^{N}}[G^{N}(t,x)]\mathrm{d} x.
\end{align}
We can take derivatives with respect to $ t $ to get
\begin{align*}
\dfrac{\mathrm{d}}{\mathrm{d} t}\BB E_{\nu^N}[M^{N}_t]
=-\int_{-\infty}^{\infty} \partial_t\BB E_{\nu^N}[G^{N}(t,x)]\mathrm{d} x.
\end{align*}
Using that
\begin{align*}
G^{N}(t,x)
=
\frac{N}{N-i_0}F^N(t,x)-\frac{i_0}{N-i_0},
\end{align*}
where $ F^N $ is the distribution function associated to the empirical law of $ Y^N $,
we have proven that
\begin{align}\label{jazmin}
w_N
=-\frac{N}{N-i_0}\int_{-\infty}^\infty \partial_t \BB E_{\mu^N}[F^N(t,x)] \mathrm{d} x
\end{align}
for every $ t>0 $,
where $ \mu^N $ is the distribution in $ [-\infty,\infty)^N $ obtained by fixing the first $ i_0 $ labels at $ -\infty $
and drawing the remaining $ N-i_0 $ ones with $ \nu^N $.

Before continuing, we need the following monotonicity result. The proof is almost a mimic of the one in Lemma 4.1 in  \cite{GJM} with the only exception that at jump times we need to couple $ k $-uples instead of pairs. The extension is straightforward.


\begin{proposition}[{\cite[Lemma 4.1]{GJM}}]\label{trikonasana}
	If $ Y $ and $ Y' $ are $ (N,p) $-BBMs satisfying
	\begin{align*}
	Y_0[i+1]-Y_0[i]
	\le_{st}
	Y_0'[i+1]- Y_0'[i]
	\end{align*}
	for every $ i\in\{1,\ldots,N-1\} $,
	then
	\begin{align*}
	Y_t[i+1]-Y_t[i]
	\le_{st}
	Y_0'[i+1]-Y_0'[i]
	\end{align*}
	for every $ i\in\{1,\ldots,N-1\}$ and every $t>0 $.
\end{proposition}
For every nonnegative function $g\colon[0,1] \to [0,\infty)$,
\begin{align}\nn
\BB E_\zeta\Big(\int_{-\infty}^{\infty}g(F^{N}(t,x))\mathrm{d} x\Big)
=\sum_{i=i_0}^{N-1}g \big(\tfrac{i}{N}\big)
\BB E_\zeta(Y^{N}_t[i+1]-Y^{N}_t[i])
\end{align}
holds for any $\zeta \in [-\infty,\infty)^N $,
so
\begin{align*}
\BB E_{\zeta'}\Big[\int_{-\infty}^{\infty}{g}(F^{N}(t,x))\mathrm{d} x\Big]
\ge
\BB E_{\zeta}\Big[\int_{-\infty}^{\infty}{g}(F^N(t,x))\mathrm{d} x\Big]
\end{align*}
if the spacings of $\zeta'$ dominate those of $\zeta $,
namely if $ \zeta'[i+1]-\zeta'[i]\ge \zeta[i+1]-\zeta[i] $ for every $ i $.
We use this fact with $g=h^N_p$ and Lemma \ref{thm_finite_hydro} to get
\begin{align*}
-\int_{-\infty}^\infty \partial_t \BB E_{\mu^N}[F^N(t,x)] \mathrm{d} x
&=
-\int \int_{-\infty}^\infty \partial_t \BB E_{\zeta}[F^N(t,x)] \mathrm{d} x \,\mu^N(\mathrm{d} \zeta)
\\ \nn
&=
\int  
\BB E_{\zeta}\Big[
\int_{-\infty}^\infty 
{h^N_p}(F^N(t,x))\Big]
\mathrm{d} x \,\mu^N(\mathrm{d} \zeta)
\\
&\ge
\BB E_{\zeta^0}\Big[
\int_{-\infty}^\infty 
{h^N_p}(F^N(t,x))
\mathrm{d} x \Big]
\end{align*}
for $ \zeta^0$  defined as
\begin{align}\nn
\zeta^0(i)
=
\begin{cases}
-\infty & \mbox{if $ i\le i_0 $}
\\
0 & \mbox{if $ i> i_0 $}.
\end{cases}
\end{align}
This inequality and \eqref{jazmin} reduces
\eqref{mora} to proving that
\begin{align}\label{modal2}
\liminf_{N\to\infty}\,\BB E_{\zeta^0}\Big[
\int_{-\infty}^\infty 
{h^N_p}(F^N(t_0,x))
\mathrm{d} x \Big]
\ge c^*
\end{align}
holds for some $ t_0>0 $. Let $ U^0 $ be the solution to the F-KPP equation \eqref{kpp_eq_1} with initial condition given by the heavyside function $ \mathbbm{1}_{[0,\infty)} $,
let $ M_t $ be the median of $ U^0(\cdot,t) $,
and let $ W_{c^*} $ be the minimal velocity wavefront.
From \eqref{coronavirus} and \eqref{kpp_property_1} 
it follows that,
given $ \varepsilon>0 $, we can fix $ t_0>0 $ and $ R>0 $ such that
\begin{align*}
\int_{M_{t_0}-R}^{M_{t_0}+R}h_p(U^0(t_0,x))\mathrm{d} x
\ge
c^* -\varepsilon.
\end{align*}
Under this choice, the expectation on the l.h.s. of \eqref{modal2} is bounded from below by
\begin{align*}
\BB E_{\zeta^0}\Big(\int_{M_{t_0}-R}^{M_{t_0}+R}
{h_p^N}(F^{N}(t,x))\mathrm{d} x\Big)
-\int_{M_{t_0}-R}^{M_{t_0}+R}h_p(U^0(t_0,x))\mathrm{d} x
+c^*-\varepsilon.
\end{align*}
The difference between the first two terms is less or equal than
\begin{align*}
\nonumber &\int_{M_{t_0}-R}^{M_{t_0}+R}
\BB E_{\zeta^0} (| {h_p^N}(F^N(t_0,x))  -h_p(F^N(t_0,x)) |)
\mathrm{d} x 
\\ 
& \quad + \int_{M_{t_0}-R}^{M_{t_0}+R}
\BB E_{\zeta^0} (| h_p(F^N(t_0,x))  -h_p(U^0(t_0,x)) |)
\mathrm{d} x.
\end{align*}
The first term vanishes due to the uniform convergence $h_p^N \to h_p$. Since $ h_p $ is Lipschitz continuous, the second term is less or equal than
\begin{align*}
2CR\,
\BB E_{\zeta^0} (\| F^N(t_0,\cdot)-U^0(t_0,\cdot) \|_\infty),
\end{align*}
for some $C>0$, which vanishes due to Theorem \ref{thm_hydro}.
This completes the proof of Proposition \ref{uttanasana}.

\section{Proof of Theorem \ref{thm_main}}

We now prove the asymptotic behavior of the velocities for arbitrary $b$ and $D$, Theorem \ref{thm_main}. The upper bound is easily obtained in terms of a (rate $ 1 $) BBM. 
If $\tilde X^N $ is a BBM and $ X^N $ is a $(b,D)$-BBM satisfying the hypotheses of Theorem \ref{thm_main} and $ X^N_0\le_{st}  \tilde X_0^N $, then $ X^N\le_{st}  \tilde X^N $. 
In fact, observe that $ \tilde X^N $ is a $ (\tilde  b,\tilde{\mathbf{D}}) $-BBM with  $  \tilde b \equiv 1$ and $  \tilde D_j= \delta_{-\infty}$ for every $ j\in\BB N $. Hence this is an immediate consequence of Proposition \ref{thm_comparison}. Since the rightmost particle of $\tilde X^N$ has velocity $\sqrt 2$, this readily implies
\begin{align*}
\limsup_{t\to\infty}\frac{1}{t}X^N_t[N_t]\le \sqrt{2} \quad \mbox{a.s.} 
\end{align*}

To prove the lower bound, we first deal with case (1), $ D( -\infty)=0 $.
Let $Y^{N} $ be an $(N,p)$-BBM with parameters $p(i,j)=\mathbbm{1} \{i=k-1,j=k\}$, $\lambda = 1/k$, and initial distribution $ Y^{N}_0 {=} X_{0}^{N} $.
Proposition \ref{bailey} implies that for $ k $ large enough $Y^N \le_{st} X^N$, so
\begin{align*}
v_N=\lim_{t\to\infty}\frac{1}{t} X^N_t[N]
\ge
\lim_{t\to\infty} \frac{1}{t}  Y_{t}^{N}[N] = 	w_N
\quad
\mbox{a.s.},
\end{align*}
the existence of the first limit following as in the proof of Proposition \ref{caramba}.
Observe that in this case we have $h_p'(v)\ge h_p'(1) = -1$ for every $ v\in (0,1) $, so the minimal velocity of the F-KPP equation with source $h_p$ is $ c^*=\sqrt{2} $ (see Section \ref{appendix1}). 
It only remains to let $N\to \infty $ and use Proposition \ref{uttanasana} to get the result.

For case (2), $ D( -\infty)>0 $, for any $\hat N \in \N$, we define the stopping time $ \tau $ by 
\begin{align*}
\tau=\inf\{t\ge 0: N_t= \hat  N \},
\end{align*}
which is finite almost surely. Let $ \hat X^{\hat N} $ be a $ (b,D) $-BBM with random initial distribution $ \hat X_0^{\hat N}\stackrel{d}{=}X_{\tau}^{ N} $. The strong Markov property guarantees
\begin{align}\label{thomson}
\liminf_{t\to\infty}\frac{1}{t} X^N_t[N_t]
=
\liminf_{t\to\infty} \frac{1}{t} \hat X_{t}^{\hat N}[\hat N_t]
\quad
\mbox{a.s.}
\end{align}
Let $Y^{\hat N} $ be an $(\hat N,p)$-BBM as before (but with $\hat N$ particles intead of $N$), and initial distribution $ Y^{\hat N}_0 \stackrel{d}{=}  \hat X_{0}^{\hat N} $.
The right-hand side of \eqref{thomson} is bounded from below by
\begin{align}\nn
\lim_{t\to\infty}\frac{1}{t}Y_t^{\hat N}[\hat N]=w_{\hat N}.
\end{align}
Since $ \hat N $ is arbitrary, Proposition \ref{uttanasana} gives the desired bound. 

\appendix

\section{F-KPP equation}\label{appendix1}

The results presented here are standard in the theory of non-linear parabolic equations; see  for example  \cite{Bra83,VVV00,GK04}.

\begin{definition}	
	Let $ t>0 $, and let $ V_0,h:\BB R\to\BB R $ and $ g:(0,t]\times\BB R \to\BB R$ be arbitrary functions.
	A (classical) solution to the differential equation
	\begin{align}\label{matambre}
	&\partial_t V=\tfrac{1}{2}\partial_{xx}V-h(V)+g 
	\\[5pt]\label{pizza}
	&V(0,\cdot)=V_0
	\end{align}
	in the time interval $ [0,t] $ is a function $ V:[0,t]\times \BB R\to\BB R $
	that satisfies the following conditions:
	\begin{enumerate}
		\item  $ V|_{\{0\}\times\BB R}=V_0 $;
		\item\label{perico} $ V|_{(0,t]\times\BB R }\in C^{1,2}((0,t]\times \BB R) $ and 
		\eqref{matambre} is satisfied for every $(s,x)\in (0,t]\times \BB R $;
		\item $ \lim_{s\downarrow 0}V(s,x)=V_0(x) $ for every $ x\in\BB R $ continuity point of $ V_0 $.
	\end{enumerate}
\end{definition}

In the previous definition,
condition
$ V|_{(0,t]\times\BB R }\in C^{1,2}((0,t]\times \BB R) $
means that there exist an open set $ A\subset\BB R^2 $ containing $ (0,t]\times\BB R $
and an extension $ \bar V\in C^{1,2}(A)  $
of $ V|_{(0,t]\times\BB R } $.

\begin{theorem}\label{thm_termotanque}
	Let $ t>0 $. Assume $ h\colon[0,1] \to \R$ is continuous with $h(0)=h(1)=0$. If $ V_0 $ is a distribution function
	and $ g:(0,t]\times \BB R\to\BB R $ is continuous and bounded, the differential equation (\ref{matambre},\ref{pizza}) has 	a unique bounded solution $ V $ in the time interval $ [0,t] $.
\end{theorem}

In the case $ h(v)=\lambda Nv $, the solution given in Theorem \ref{thm_termotanque} is characterized by the integral representation
\begin{align}\nonumber
V(s,x)=
\int_{-\infty}^{\infty} H(s,x-y)V_0(y)\mathrm{d} y
+\int_0^s \int_{-\infty}^{\infty} H(s-s',x-y)g(s',y)\mathrm{d} y\mathrm{d} s',
\end{align}
being $ H $ the Green kernel associated to operator $ \partial_t-\tfrac{1}{2}\partial_{xx}+\lambda N $, that is 
\begin{align}\nonumber
H(s,x)= \Phi(s,x)\, e^{-\lambda Ns}.
\end{align}
The function $ \Phi $ is defined in \eqref{cuadrito}. Next we state a result that controls the stability of the solution under perturbations of the initial condition and the function $ g $.

\begin{theorem}\label{thm_quinoto}
	Let $ t>0 $ and assume $ h\colon[0,1] \to \R$ is continuous with $h(0)=h(1)=0$. For every $ M>0 $ there exists a constant $ C=C(t,M)>0 $ 
	such that
	\begin{align}\nonumber
	\| V(t)-\tilde V(t) \|_\infty\le C\big(  \| V_0-\tilde{V}_0 \|_\infty
	+\| g-\tilde g \|_\infty\big)
	\end{align}
	for every $ V_0,\tilde{V_0} $ distribution functions and every $ g ,\tilde g\in C((0,t]\times \BB R)$ such that $ 0\le g,\tilde g\le M $.
	Here $ V  $ [resp. $ \tilde V $] is the unique bounded solution to equation (\ref{matambre}-\ref{pizza}) in the time interval $ [0,t] $ associated to  $ V_0$ and $g $ (resp. $ \tilde{V}_0 $ and $ \tilde g $).
\end{theorem}

Suppose during the rest of the section that $ g\equiv 0 $, 
and that $ h\in C^1([0,1]) $ satisfies $ h(u)>0 $ for every $ u\in (0,1) $.
A traveling wave with speed $  c \in\BB R $ is a solution to equation \eqref{matambre} of the form $ U(t,x)=W_ c (x- c  t) $ with $ W_ c  \in C^2(\BB R)$ non-decreasing and satisfying 
$ W_ c (-\infty)=0 $, $W_ c (\infty)=1 $.
The function $ W_ c  $ is called a wavefront  and is characterized by satisfying the ODE
\begin{align*}
\frac{1}{2}W_ c ''+ c  W_ c '-h(W_ c )=0.
\end{align*}
The following facts are well known:
\begin{enumerate}
	\item There exists a minimal speed $c^*>0$. More precisely, for each $  c \ge  c ^* $ there is a (unique) wavefront $ W_ c  $ with speed $ c  $, and there are no wavefronts for $ c<c^* $.
	\item For each $  c \ge c ^* $ we have,
	\begin{align}\label{coronavirus}
	c =\int_{-\infty}^\infty h(W_ c )\mathrm{d} x .
	\end{align}
	\item $  c ^*\ge \sqrt{-2h'(1)} $ and identity holds if $  h'(1)<0 $ and $ h'(u)\ge h'(1) $ for every $ u\in [0,1] $.
	\item  If $ U^0 $ is the unique bounded solution to equation \eqref{matambre} with initial value the heavyside function $ \mathbbm{1}_{[0,\infty) }$,
	and $ M_t $ is the median of $ U^0(t,\cdot) $,
	then
	\begin{align}\label{kpp_property_1}
	\lim_{t\to\infty}\| U^0(t,\cdot +M_t) -W_{ c ^*} \|_\infty =0.
	\end{align}
\end{enumerate}

\section{Propagation of chaos in the \texorpdfstring{$ (N,p) $}{TEXT}-BBM}\label{proof_correlations}

We prove here Lemma \ref{caramelo}. The key tool is the construction and control of the clans of ancestors. With that in mind, we introduce an alternative graphical construction of the $ (N,p) $-BBM.
We first describe it in words. Every index $ i\in\{ 1,\ldots,N \} $ rings at rate $ \lambda k $. When it rings, a $ (k-1) $-tuple of indices $ \{i_1,\ldots,i_{k-1}\}\subset \{ 1,\ldots,N \}\setminus\{i\} $ and a pair  $ (a,b)\in \{ (a',b'):1\le a'<b'\le k \} $ are chosen,
the first one uniformly at random and the second one according to $p$.
Let $j_1<\ldots < j_k$ be the ordered $k$-tuple $\{i_1,\ldots,i_{k-1},i\}$.
If $ Y^N[j_a]=Y^N(i) $, the operation $ \Gamma_{ j_a,j_{b}} $ (defined just before Section \ref{sctn_comparison}) is applied; otherwise, nothing happens.
Between time marks, the particles diffuse as independent Brownian Motions. 

To convince the reader that the two given constructions of the $ (N,p) $-BBM coincide, we compute the rate at which a particle with quantile $ i $ performs a jump in both cases.
In the one given just before Section \ref{sctn_comparison},
a $ k $-tuple is chosen at rate $ \lambda N $, the quantile $ i $ is at position $ u $ in this $ k $-tuple with probability
$\binom{i-1}{u-1}\binom{N-i}{k-u}\binom{N}{k}^{-1}$,
and it performs a jump with probability $ \sum_{v=u+1}^kp(u,v) $,
so the desired rate in this case is
\begin{align}\label{lllooo1}
\lambda N\sum_{u=1}^{k-1}\frac{\binom{i-1}{u-1}\binom{N-i}{k-u}}{\binom{N}{k}}
\sum_{v=u+1}^kp(u,v).
\end{align}
In the construction given in the previous paragraph, quantile $ i $ rings at rate $ \lambda k $,
it is located at position $ u $ in the $ k $-tuple $ \{j_1,\ldots,j_k\} $ with probability $\binom{i-1}{u-1}\binom{N-i}{k-u}\binom{N-1}{k-1}^{-1}$,
and performs a jump with probability $ \sum_{v=u+1}^kp(u,v) $,
hence 
\begin{align}\label{lllooo2}
\lambda k\sum_{u=1}^{k-1}\frac{\binom{i-1}{u-1}\binom{N-i}{k-u}}{\binom{N-1}{k-1}}
\sum_{v=u+1}^kp(u,v)
\end{align}
is the corresponding rate here. Expressions \eqref{lllooo1} and \eqref{lllooo2} coincide.

An important observation about this new approach, that will be used later, is that a necessary condition for the $ i $-th particle to jump is that the $ i $-th Poissonian clock has rang.

The process is then obtained, as in the first construction, as a deterministic function of an $N$-dimensional Brownian Motion 
$B=(B(1),\ldots,B(N))=\{(B_t(1),\ldots,B_t(N)):t\ge 0\} $ 
and a marked Poisson Process $C= \cup_{i\in\{ 1,\ldots,N \}} C^i $. Here $ C^i= \{ (T_m^i,S_m^i,(a^i_m,b^i_m),i):m\in\BB N \} $. For every $i$,  $ \{ T_m^i:m\in\BB N \} $ is a Poisson Process in $ [0,\infty) $ with intensity $ \lambda k $; for every  $ m $, $ S_m^i \subset \{ 1,\ldots,N \}\setminus\{i\}$  is a $ (k-1) $-tuple uniformly chosen at random and $ (a^i_m,b^i_m) $ a random pair with distribution $ p $.
The possible-jump events occur at the superposition of the Poissonian times $ \{ T_u:u\in\BB N \} =\cup_{i\in\{ 1,\ldots,N \}}\{ T_m^i:m\in\BB N \}$. We use possible-jump because it could happen that no jump is performed at these times.
Between these times, the increments of the particle  labeled $ i $ are governed by $ B(i) $.
At time  $ T_m^i $,
$ S_m^i $ and $ (a_m^i,b_m^i) $ play the roles of the tuple $ \{i_1,\ldots,i_{k-1}\} $ and the pair  $ (a,b) $ introduced in the first paragraph of this section respectively.

For each index  $ j\in\{ 1,\ldots,N \} $ we construct, as a deterministic function of $ C $, an auxiliary process $ \{ \varphi_t(j):t\ge 0 \} $ that we call the \textit{forward} clan of ancestors. This process is Markovian 
and its state space is the family of subsets of $ \{ 1,\ldots,N \} $.
For $ s\in [0,T_1) $, define $ \varphi_s(j)=\{j\} $. Suppose we have defined $ \varphi_s(j) $ for $ s\in [0,T_u) $,
and let $ T_u=T_m^i $.
If $ i\notin \varphi_{T_u-}(j) $, do nothing: $ \varphi_s(j)=\varphi_{T_u-}(j) $ for every $ s\in [T_u,T_{u+1}) $.
If instead $ i\in \varphi_{T_u-}(j) $, define $ \varphi_s(j)=\varphi_{T_u-}(j) \cup S^i_m$ for every $ s\in [T_u,T_{u+1}) $.

For $ t\ge 0 $ and $ i\in\{ 1,\ldots,N \} $, let
$ C^i(t) = \{(T^i_m,S^i_m,(a^i_m,b^i_m),i):m\in\BB N\textnormal{ such that }T^i_m\le t\} $
be the projection of $C^i$ on the time interval $ [0,t] $.
For every $ j\in\{ 1,\ldots,N \} $, $ \varphi_t(j) $ is a deterministic function of the Poisson marks $  C(t)=\cup_{i\in\{ 1,\ldots,N \}}C^i(t) $. We emphasize this by writing $ \varphi_t(j)=\varphi_t(j)[C(t)] $.
Let $ R_t:[0,t]\to [0,t] $ be the reflection $ R_ts=t-s $.  Define also
$$ R_tC^i(t)=\{(R_ts,S,(a,b),i):(s,S,(a,b),i)\in C^i(t)\} $$ and
$$ R_tC(t)=\bigcup_{i\in\{ 1,\ldots,N \}}R_t C^i(t). $$
Finally, for every $ j\in\{ 1,\ldots,N \}$, the \textit{set of ancestors} $ \psi_t(j) $ is defined by  
$$\psi_t(j):= \varphi_t(j)[R_tC(t)]. $$
The process
$ \{\psi_t(j):t\ge 0\} $ is not Markovian,
and $ \psi_t(j) $ represents the set of indices 
of the particles that could have had influence in 
$ Y^N_t(j) $.
The clan of ancestors has been used before, for instance in \cite{AFG16,GJM}. We refer to those references for more details on this construction.

We now proceed with the proof of Lemma \ref{caramelo}.

\begin{proof}[Proof of Lemma \ref{caramelo}]
	Fix $ N $, $ t $, $ x $ and $ \ell $. 
	Expanding the $\ell$-th power of $F^N$, we have
	\begin{align}\nonumber
	F^N(t,x)^\ell
	= N^{-\ell}\sum_{\substack{i_1,\ldots,i_l\in\{ 1,\ldots,N \} \\ \textnormal{all different}}} \
	\prod_{u=1}^\ell \mathbbm{1}\{ Y^N_t(i_u)\le x \}
	+N^{-l}   \mbox{\ding{172}},
	\end{align}
	where 
	$ \mbox{\ding{172}} $ is the sum of all the $ \ell$-th factors with at least one repeated index.
	In $ \mbox{\ding{172}} $, there are $ N^\ell-N(N-1)\ldots (N-(\ell-1)) $ terms, each of which is bounded in absolute value by one, so
	\begin{align}\nonumber
	|N^{-\ell} \mbox{\ding{172}} | \le 1-\tfrac{N-1}{N}\tfrac{N-2}{N}\ldots \tfrac{N-(\ell-1)}{N}
	\eqqcolon a_{N,\ell}.
	\end{align}
	Analogously, for fixed $ \zeta\in [-\infty,\infty)^N $,
	\begin{align}\nonumber
	U_{\zeta}^N(t,x)^\ell
	=N^{-\ell}\sum_{\substack{i_1,\ldots,i_\ell\in\{ 1,\ldots,N \} \\ \textnormal{all different}}} \
	\prod_{u=1}^\ell \BB P_\zeta\big( Y^N_t(i_u)\le x \big)
	+N^{-\ell} \mbox{\ding{173}}
	\end{align}
	with $ |N^{-\ell} \mbox{\ding{173}}|\le a_{N,\ell} $.
	Then
	\begin{align}\nonumber
	&\big|\BB E_\zeta  [F^N(t,x)^\ell]  - U^N_{\zeta}(t,x)^\ell \big| \le 2a_{N,\ell}
	\\ \nonumber  & \quad + N^{-\ell}\sum_{\substack{i_1,\ldots,i_\ell\in\{ 1,\ldots,N \} \\ \textnormal{all different}}} \,
	\Big|
	\BB P_\zeta \Big( \bigcap_{u=1}^\ell  [Y^N_t(i_u)\le x ] \Big)
	-\prod_{u=1}^\ell  \BB P_\zeta\{ Y^N_t(i_u)\le x \}
	\Big|.
	\end{align}
	
	We next prove that, for distinct indices $ i_1,\ldots,i_\ell $,
	\begin{align}\label{camarote}
	\Big|
	\BB P_\zeta \Big( \bigcap_{u=1}^\ell  [Y^N_t(i_u)\le x ] \Big)
	-\prod_{u=1}^\ell  \BB P_\zeta( Y^N_t(i_u)\le x )
	\Big|
	\le
	\frac{k^2(e^{2\lambda k(k-1)t}-1)}{N-1}.
	\end{align}
	The last inequality together with the fact that $ a_{N,\ell}\sim CN^{-1} $  will allow us to conclude.
	(To see that $ a_{N,\ell}\sim CN^{-1} $ one can easily prove by induction that $Na_{N,\ell}\le \frac{\ell(\ell-1)}{2}$ for every $ \ell\ge 2 $.)
	
	Define the event
	\begin{align}\nonumber
	\CAL I[i_1,\ldots,i_\ell]
	=
	\bigcup_{\substack{m,n\in\{1,\ldots, \ell\} \\ m\neq n }}
	[\psi_t(i_m)\cap \psi_t(i_n)\neq \emptyset ],
	\end{align}
	namely the   complement of  $ \CAL I[i_1,\ldots,i_\ell]$ occurs when the clans of ancestors are pairwise disjoint. 
	On the one hand,
	\begin{align}\nonumber
	\BB P_\zeta \Big( \bigcap_{u=1}^\ell   [Y^N_t(i_u)\le x ] \Big)
	&= \BB P_\zeta
	\Big[  \Big( \bigcap_{u=1}^\ell   [Y^N_t(i_u)\le x ] \Big) \cap \CAL I(i_1,\ldots,i_\ell) \Big]
	\\
	\nonumber	&\blanco{=}+\sum_{A_1,\ldots,A_\ell }^* \BB P_\zeta
	\Big[
	\bigcap_{u=1}^\ell  [Y^N_t(i_u)\le x,\psi_t(i_u)=A_u] 
	\Big]
	\\ \nonumber
	&= \BB P_\zeta
	\Big[  \Big( \bigcap_{u=1}^\ell   [Y^N_t(i_u)\le x ] \Big) \cap \CAL I(i_1,\ldots,i_\ell) \Big]
	\\[5pt] \label{escape}
	&\blanco{=}+\sum_{A_1,\ldots,A_\ell }^* \,
	\prod_{u=1}^\ell  \,
	\BB P_\zeta
	[Y^N_t(i_u)\le x,\psi_t(i_u)=A_u].
	\end{align}
	The symbol $ \displaystyle\sum_{A_1,\ldots,A_\ell}^* $
	means that we are summing over subsets $ A_1,\ldots,A_\ell\subset\{ 1,\ldots,N \} $ that are pairwise disjoint and such that $ i_u\in A_u $ for every $ u\in\{1,\ldots,\ell\} $.
	In the last identity, 
	we used the factorization property of the clans of ancestors
	\begin{align}\nonumber
	\BB P_\zeta
	\Big[
	\bigcap_{u=1}^\ell  [Y^N_t(i_u)\le x,\psi_t(i_u)=A_u] 
	\Big]
	=
	\prod_{u=1}^\ell  \,
	\BB P_\zeta
	[Y^N_t(i_u)\le x,\psi_t(i_u)=A_u],
	\end{align}
	that holds because, for every $ u\in\{1,\ldots,l\} $,
	the event $ [Y^N_t(i_u)\le x,\psi_t(i_u)=A_u]  $
	is measurable with respect to the $ \sigma $-algebra generated by $\displaystyle \bigcup_{r\in A_u}\{ C^r,B(r) \} $.
	
	We now work with the second term inside the absolute value in \eqref{camarote},
	\[
	\prod_{u=1}^\ell  \BB P_\zeta( Y^N_t(i_u)\le x ).
	\]
	Consider $\ell$ independent copies $ \{ (B^{(u)},C^{(u)}):u\in\{1,\ldots,\ell\}\} $ of $ (B,C) $, and let $ Y^{N,(u)} $ be the process constructed as a function of $ (B^{(u)},C^{(u)}) $, all the copies with initial condition $ \zeta $.
	Similarly,
	let $ \psi^{(u)}=(\psi^{(u)}(1),\ldots,\psi^{(u)}(N)) $ be the process $ \psi $ constructed as a function of $ C^{(u)} $,
	and let
	\begin{align*}\nonumber
	\CAL I^\otimes[i_1,\ldots,i_\ell]
	=
	\bigcup_{\substack{m,n\in\{1,\ldots,\ell\} \\ m\neq n }}
	[\psi^{(m)}_t(i_m)\cap \psi^{(n)}_t(i_n)=\emptyset ].
	\end{align*}
	Then
	\begin{align}
	\nonumber	\prod_{u=1}^l \BB P_\zeta\big( Y^N_t(i_u)\le x \big)
	= \ & \BB P_\zeta
	\Big[
	\bigcap_{u=1}^\ell [Y^{N,(u)}_t(i_u)\le x]
	\Big]
	\\ \label{enter}
	= \ &
	\BB P_\zeta
	\Big[
	\Big(
	\bigcap_{u=1}^\ell [Y^{N,(u)}_t(i_u)\le x]
	\Big)
	\cap
	\CAL I^\otimes[i_1,\ldots,i_\ell]
	\Big]
	\\ \nonumber 
	& \hspace{-2pt} +
	\sum_{A_1,\ldots,A_\ell}^*
	\BB P_\zeta
	\Big[
	\bigcap_{u=1}^\ell \big(Y^{N,(u)}_t(i_u)\le x,
	\psi^{(u)}_t(i_u)=A_u
	\big)
	\Big].
	\end{align}
	Since \eqref{escape} and \eqref{enter} coincide,
	and since
	\begin{align}\nonumber
	\BB P
	( \CAL I[i_1,\ldots,i_\ell] )
	=
	\BB P
	(
	\CAL I^\otimes[i_1,\ldots,i_\ell]
	)
	\end{align}
	again by the factorization property of the clans of ancestors, 
	the left-hand side of \eqref{camarote} is bounded by
	$2\BB P
	( \CAL I[i_1,\ldots,i_\ell] )
	$.
	Inequality \eqref{camarote} has been reduced to proving that
	\begin{align*}
	\BB P
	( \CAL I[i_1,\ldots,i_\ell] )
	\le 
	\frac{\lambda k^3(e^{2\lambda k(k-1)t}-1)}{2(N-1)}.
	\end{align*}
	
	Since the growth rate of $ | \varphi_s(j) | $ is bounded from above by $ \lambda k(k-1) | \varphi_s(j) |$,
	we have
	\begin{align*}
	\BB E\big(|\varphi_s(j)|\big) \le e^{\lambda k(k-1)s},
	\end{align*}
	for every $ s\ge 0 $.
	We examine now the rate at which the indicator function of the event
	\begin{align}\nonumber
	\CAL J_s[i_1,\ldots,i_\ell]
	= \bigcup_{\substack{m,n\in\{1,\ldots,\ell\} \\ m\neq n }}
	[\varphi_s(i_m)\cap \varphi_s(i_n)=\emptyset ]
	\end{align}
	jumps from zero to one. If such a jump occurs at time $s$, then $ \varphi_{s-}(i_1),\ldots,\varphi_{s-}(i_\ell) $ are pairwise disjoint and there are $ m,n\in\{1,\ldots,\ell\} $, $ m\neq  n$,
	such that an index $ u\in A_m $ rings
	and the chosen $ (k-1) $-tuple contains some $ v\in A_n $.
	Under these considerations, we conclude that this rate is bounded from above by
	\begin{align}
	\nonumber &\sum_{A_1,\ldots,A_\ell}^* 
	\BB P[\varphi_s(i_1)=A_1,\ldots,\varphi_s(i_\ell)=A_\ell]
	\sum_{\substack{m,n\in\{1,\ldots,\ell\} \\ m\neq n  }} \,
	\sum_{u\in A_m, v\in A_n}\lambda k\frac{k-1}{N-1}
	\\ \label{mostaza}
	&\quad=
	\lambda k\frac{k-1}{N-1} \,
	\sum_{\substack{m,n\in\{1,\ldots,\ell\} \\ m\neq n  }} \,
	\sum_{A_1,\ldots,A_\ell}^* \,
	\sum_{u\in A_m, v\in A_n} \,
	\BB P\big[\varphi_s(i_1)=A_1,\ldots,\varphi_s(i_\ell)=A_\ell\big].
	\end{align}
	Fix a pair $ m,n\in\{ 1,\ldots,N \} $, $ m\neq n $. Without loss of generality, we assume $ m=1 $, $ n=2 $.
	We have
	\begin{align}\nonumber
	&\sum_{A_1,\ldots,A_\ell}^* \,
	\sum_{u\in A_1,v\in A_2}
	\BB P[\varphi_s(i_1)=A_1,\ldots,\varphi_s(i_\ell)=A_\ell]
	\\[5pt]\nonumber
	&\quad =
	\sum_{A_1,A_2}^* 
	|A_1||A_2|
	\BB P[\varphi_s(i_1)=A_1,\varphi_{i_2}(s)=A_2]
	\, \mbox{\ding{174}},
	\end{align}
	where
	\begin{align}\nonumber
	\mbox{\ding{174}} = \sum_{A_3,\ldots,A_\ell}^*
	\BB P[\varphi_s(i_3)=A_3,\ldots,\varphi_s(i_\ell)=A_\ell].
	\end{align}
	Using that $ \mbox{\ding{174}}\le 1 $, that
	\begin{align}\nonumber
	\sum_{A_1,A_2}^* 
	|A_1||A_2|
	\BB P[\varphi_s(i_1)=A_1,\varphi_s(i_2)=A_2]
	=\BB E\big(|\varphi_1(s)|\big)^2\le e^{2\lambda k(k-1)s}
	\end{align}
	and pluggin in \eqref{mostaza},
	we obtain that \eqref{mostaza} is bounded by
	$\lambda k\tfrac{k-1}{N-1}e^{2\lambda k(k-1)s}k^2$. 
	Finally, using that the distribution of the Poisson point process in $ [0,t] $ is invariant under the reflection $ R_t $,
	\begin{align}\nonumber
	\BB P
	\big( \CAL I[i_1,\ldots,i_\ell] \big)
	&=\BB P\big( \CAL J_t[i_1,\ldots,i_\ell]\big)
	=\int_0^t   \frac{\mathrm{d}}{\mathrm{d} s}  
	\BB E
	[\mathbbm{1}\{\CAL J_s[i_1,\ldots,i_\ell]\}]
	\mathrm{d} s
	\\[5pt] \nonumber
	&\le \int_0^t \lambda k  \tfrac{k-1}{N-1}e^{2\lambda k(k-1)s}k^2\mathrm{d} s
	=\lambda k\frac{k^2(e^{2\lambda k(k-1)t}-1)}{2(N-1)}.
	\qedhere
	\end{align}
\end{proof}

\bibliographystyle{plain}
\bibliography{bdbbm_biblio}

\end{document}